\documentclass[a4paper,10pt]{amsart}
\usepackage{amssymb,amsmath, pinlabel, mathabx}

  \usepackage{hyperref}
\usepackage{mathtools}

\input xy
\xyoption{all}

\usepackage{color}

\newtheorem{theorem}{Theorem}[section]
\newtheorem{lemma}[theorem]{Lemma}
\newtheorem{proposition}[theorem]{Proposition}
\newtheorem{corollary}[theorem]{Corollary}

\theoremstyle{definition}
\newtheorem{definition}[theorem]{Definition}

\newtheorem{remark}[theorem]{Remark}
\newtheorem{example}[theorem]{Example}

\numberwithin{equation}{section}

\newcommand{\gras}[1]{{\mathbb #1}} 

\newcommand{\N}{\gras{N}}
\newcommand{\Z}{\gras{Z}} 
\newcommand{\Q}{\gras{Q}} 
\newcommand{\R}{\gras{R}} 
\newcommand{\C}{\mathbb{K}}

\newcommand{\cE}{\mathcal{E}}

\newcommand{\Supp}{{\mathcal{S}}}
\newcommand{\Irr}{{\mathrm{Irr}}}
\newcommand{\Ess}{{\mathrm{ess}}}

\def\elem(#1,#2){  \{ \frac{#1}  {\overline {\ #2\ }} \} }

\title[Variations on inversion theorems for Newton-Puiseux series]
{Variations on inversion theorems for Newton-Puiseux series  }
\author{Evelia Rosa Garc\'{\i}a Barroso}
\address{Departamento de Matem\'aticas, Estad\'{\i}stica e I.O.
Secci\'on de Matem\'aticas, Universidad de La Laguna. Apartado de Correos 456.
38200 La Laguna, Tenerife, Espa\~na.}
   \email{ergarcia@ull.es}
\author{Pedro Daniel Gonz\'alez P\'erez} 
\address{Instituto de Ciencias Matem\'aticas (CSIC-UAM-UC3M-UCM), Departamento de \'Algebra , Facultad de Ciencias Matem\'aticas,
Universidad Complutense de Madrid, Plaza de las Ciencias 3, Madrid 28070, Espa\~na.}
   \email{pgonzalez@mat.ucm.es}
\author{Patrick Popescu-Pampu}
   \address{Univ. Lille, UMR 8524, Laboratoire Paul Painlev\'e, F-59000 Lille, France.}
   \email{patrick.popescu@math.univ-lille1.fr}
\date{3 December 2016}


\begin{document}

\begin{abstract}
    Let $f(x,y)$ be an irreducible formal power series without constant term, over 
    an algebraically closed field of characteristic zero. 
    One may solve the equation $f(x,y)=0$ by choosing either $x$ or $y$ 
    as independent variable, getting two finite sets of Newton-Puiseux series. 
    In 1967 and 1968 respectively, Abhyankar and Zariski published proofs 
    of an \emph{inversion theorem}, expressing the \emph{characteristic exponents} of one set 
    of series in terms of those of the other set. In fact, a more general theorem, 
    stated by Halphen in 1876 and proved by Stolz in 1879,     
    relates also the \emph{coefficients} of the characteristic terms of both sets of series. 
    This theorem seems to have been completely forgotten. We give 
    two new proofs of it and we generalize it to a theorem concerning 
    irreducible series with an arbitrary number of variables. 
\end{abstract}

\maketitle

\begin{center}
    {\bf This paper appeared online in Mathematische Annalen on 3.12.2016. 
         The final publication is available at Springer via http://dx.doi.org/10.1007/s00208-016-1503-1}
\end{center}

\tableofcontents

\section{Introduction}
\label{sec:intro}

Let $f(x,y)$ be a polynomial with complex coefficients and without constant term. 
In his \emph{Treatise of fluxions and of infinite series} \cite[Sect. XXIX-XXXIII]{N 36}, 
Newton described 
an iterative method to compute a  formal power series $\eta \in \mathbb{C}[[t]]$, such that 
$f(x, \eta(x^{1/n})) =0$ for certain positive integer $n$. In his 1850 paper \cite{P 50}, 
Puiseux  proved that the series produced  
by Newton's algorithm were convergent whenever one starts from 
a convergent series $f(x,y) \in \mathbb{C}[[x, y]]$.  Since then, the series whose 
exponents are positive rational  numbers with bounded 
denominators, be they convergent or not, are called either \emph{Newton-Puiseux 
series} or \emph{Puiseux series}. In the sequel we  will use the first denomination. 

In the years 1870, Smith \cite{S 73}  and Halphen \cite{H 76} realized that for several 
questions about the singularities of plane algebraic curves, a finite number 
of the exponents of a Newton-Puiseux 
series were more important than the others. Halphen called those special exponents 
\emph{characteristic}. Their modern definition is the following one: if one looks 
at the sequence of exponents taken in increasing size, then the characteristic ones  
are exactly those at which jumps the lowest common denominator which may be used for the 
exponents up to that point.  It is a basic fact that if $f(x,y) \in \mathbb{C}[[x, y]]$ is irreducible, 
then all the associated Newton-Puiseux series have the same sequence of characteristic exponents.
Let us call it \emph{the characteristic sequence of $f(x,y)$ relative to $x$}. 

Nowadays, one describes usually the importance of this notion as follows. 
Consider a \emph{branch} $C$ (that is, an irreducible germ of curve) 
on a germ of smooth complex analytic surface $S$. If $(x,y)$ are local 
coordinates on $S$ and $f(x,y) \in \mathbb{C}[[x,y]]$ is a defining function of 
$C$ relative to those coordinates, then one may consider its characteristic sequence relative 
to $x$. This sequence is independent of $(x,y)$ once  the $y$-axis is 
transversal to the branch $C$, that is, once its tangent does not coincide with the reduced 
tangent cone of $C$, which is a line. One speaks then of \emph{the generic characteristic 
sequence of $C$}. Its main property is that \emph{it is a complete 
invariant of the embedded topological type of the branch in the ambient 
germ of surface} (see for instance \cite[Theorem 5.5.8]{W 04}).   
In fact, most computations of other topological invariants of the pair $(S, C)$ are done 
in terms of its generic characteristic sequence.

It is nevertheless important to work also with Newton-Puiseux series computed 
relative to \emph{non-generic} coordinate systems. For instance, another usual way to study 
the branch $C$ is to perform its process of embedded 
resolution by blow-ups of points (see for instance \cite[Chapter 3]{W 04}).  A basic 
problem is then to express the generic characteristic sequence of the 
strict transform of $C$  obtained after one blow up in terms of that of $C$. 
If one starts from a generic Newton-Puiseux series  $\eta(x^{1/n})$ of $C$, 
then $x^{-1} \cdot \eta(x^{1/n})$ is a Newton-Puiseux series of the strict transform of 
$C$. This series is generic for the strict transform if and only if the $x$-order 
of $\eta(x^{1/n})$ is at least $2$. In this case, it is immediate to get from it the 
generic characteristic sequence of the strict transform. But how to proceed when 
this is not the case?

One gets the following problem, in whose formulation we replaced for simplicity 
the strict transform by the initial branch: to 
compute the generic characteristic exponents of $C$ in terms of those of a Newton-Puiseux 
series  $\eta(x^{1/n})$ of a defining function $f(x,y)$, when $C$ is tangent to the $y$-axis. 
But in this case $C$ is necessarily transversal to the $x$-axis. Therefore, if 
$\xi(y^{1/m})$ is a Newton-Puiseux series of $f(x,y)$ relative to $y$, that is, if  $f(\xi(y^{1/m}), y) =0$,  
then its characteristic sequence is exactly the generic characteristic sequence of $C$.  
Consequently, it is enough to express the characteristic 
sequence of $\xi(y^{1/m})$ in terms of that of $\eta(x^{1/n})$. 
Such an \emph{inversion theorem} (called in this way because one \emph{inverts} 
the roles of $x$ and $y$ in passing from $\eta(x^{1/n})$ 
to $\xi(y^{1/m})$) is well-known and it is often 
attributed to Abhyankar's paper  \cite{A 67} of 1967 or to Zariski's paper \cite{Z 68} of 1968. 
Proofs of this inversion theorem can be also found in  \cite[Section 5.6]{CA 00}, \cite
[Theorem 5.2.21]{DJP 00}, \cite[Proposition 4.3]{PP 03} and \cite[Page 111]{A 10}.

We were very surprised to discover that in his 1876 paper \cite[page 91]{H 76}, 
Halphen had already formulated  a stronger result than the previous inversion 
theorem. He did not provide a proof of it. As far as we know, the first 
proof was given by Stolz \cite[Sect. 3]{S 79} in 1879. For this reason, 
we will speak in the sequel about \emph{the Halphen-Stolz inversion theorem}. 
It is stronger than the inversion theorem of Abhyankar-Zariski 
because it does not only provide formulae for the characteristic 
exponents of $\xi(y^{1/m})$ in terms of those of $\eta(x^{1/n})$, 
\emph{but also for the corresponding coefficients}.  The previous papers of 
Halphen and Stolz seem to be forgotten, even though they were mentioned in 
Halphen's appendix \cite{H 84} to Salmon's 
treatise on plane algebraic curves, which is cited sometimes nowadays.

\medskip

\emph{The aim of our paper is to extend the Halphen-Stolz inversion theorem to an arbitrary 
number of variables.} We achieve this aim in Corollary \ref{invPuiseuxgen} of our 
Inversion Theorem \ref{inversioncoefgen}. 

\medskip

In order to arrive at those results, we give first 
two new proofs of the classical Halphen-Stolz 
inversion theorem (stated by us as Corollary \ref{invPuiseux} of Theorem \ref{inversioncoef}). 
The first one is based on the relations between the coefficients associated 
to the \emph{irreducible exponents}  
of an invertible power series (see Definition \ref{defindec}), 
those of its powers and those of its \emph{dual} 
(see Proposition \ref{esspowers}). A flow-chart representing our line of reasoning 
for this first proof is drawn in diagram (\ref{reasoning}). 
Our second proof uses a formula expressing \emph{all} the coefficients 
of the Newton-Puiseux series $\xi(y^{1/m})$ in terms of those of 
$\eta(x^{1/n})$ (see Proposition \ref{invexpl}). This formula,  
based on the \emph{Lagrange inversion theorem} (see Theorem \ref{Lagrinv}), 
generalizes the Halphen-Stolz inversion formula for the coefficients.  

It is the first proof which we extend into a proof of the several-variables case. We could have 
given directly the most general statements and proofs. We preferred to start explaining  
in a detailed way the classical case, because it served us as a model for building the 
general proof, and because we feel that in this way the paper is easier to read. Its title is 
inspired by the title of Griffiths' paper \cite{G 76}. 

As the generic characteristic sequence is crucial for understanding the embedded topology 
of complex plane branches, we expect that the associated coefficients 
could play an important role in 
problems related not only to their topology, but also to their analytical type. It is the main 
reason why we considered that it is important to bring to the attention of researchers 
the forgotten inversion theorem of Halphen-Stolz.

Zariski's proof of the inversion theorem in \cite{Z 68} 
was obtained as an application of the theory of saturation of 
local rings, in connexion with the study of topological equisingularity. 
There exist other notions of saturation, for instance, the Lipschitz-saturation (see \cite{PT 69, L 75}) 
and also the presaturation of Campillo, which is better adapted to 
positive characteristic (see \cite{C 83, C 88}).
In the case of irreducible germs of \emph{quasi-ordinary} hypersurface singularities,
Zariski's results on saturation and Lipman's \emph{inversion theorem}  appear also in 
the combinatorial characterization of the embedded topological type of this class of singularities 
 (see \cite{G 88, L 88A}).
We expect that these lines of research combined with our generalized inversion theorem 
will lead to a better understanding of the invariants of singularities of other classes of
hypersurface germs.

\medskip

The article is structured as follows.
In Section \ref{sec:NPThm} we recall basic facts about Newton-Puiseux series 
associated to plane branches and their characteristic exponents. 
In Section \ref{sec:irred} we introduce the notions of \emph{irreducible} 
and \emph{essential exponents} of a
series and we give some results relating the coefficients of certain pairs of invertible series. 
In Section \ref{sect:applinv} we explain our proofs  of the Halphen-Stolz 
inversion theorem and of its generalization into an inversion formula for all coefficients, 
based on a Lagrange inversion formula. Finally, in Section \ref{arbnumb} we prove our 
generalization concerning an arbitrary number of variables and we explain in which way 
it extends the inversion theorem of Lipman concerning the characteristic exponents 
of quasi-ordinary series. Note that we work always over a fixed algebraically closed field of 
characteristic zero.

\medskip
{\bf Acknowledgements.} 
   This research was partially supported 
   by the French grants ANR-12-JS01-0002-01 SUSI, Labex CEMPI 
   ANR-11-LABX-0007-01 
   and the Spanish grants  MTM2012-36917-C03-0, MTM2013-45710-C2-2-P,  
   MTM2016-76868-C2-1-P, MTM2016-80659-P.
   We are grateful to Herwig Hauser and Hana Kov\'a\v{c}ov\'a for 
  their translation of parts of Stolz' paper. The third-named author is grateful to Micka\"el 
  Matusinski for the opportunity to explain our first proof of the Halphen-Stolz 
  theorem at the Geometry seminar of the University of Bordeaux. 
  We thank him and the anonymous referee for their remarks on a previous version of this article,  
  which allowed us to improve our presentation.  We also thank Antonio Campillo for the 
  information he sent us about the uses of Abhyankar-Zariski inversion.

\section{Newton-Puiseux series and their characteristic exponents}  \label{sec:NPThm}

In this section we introduce the notations and vocabulary about power series 
with integral or rational exponents which will be used throughout the text. 
Among the series with rational exponents, we will be interested only in those 
with bounded denominators, called \emph{Newton-Puiseux series}. 
We conclude the section introducing the well-known notion 
of \emph{characteristic exponents} of a Newton-Puiseux series. 
\medskip

Throughout the paper $\mathbb{N}$ denotes the set of non-negative integers,  
$\N^*$ denotes  $\N \setminus \{0\}$ and 
 $\C$ denotes a fixed algebraically closed field of characteristic zero.

The following definition explains the basic vocabulary and notations about power 
series with integer exponents which will be used in the sequel: 

\begin{definition}  \label{def:entire}
    The ring $\C[[t]]$ of {\bf entire series} consists of the formal power series in the 
    variable $t$, with coefficients in $\C$ and 
    exponents in  $\N$. We say that the elements of its field of fractions $\C((t))$ are 
    {\bf meromorphic series}.  They are exactly the series with coefficients in $\C$, 
    exponents in $\Z$ and a finite number of terms with negative exponents. 
    If $\eta \in \C((t))$ and $m \in \Z$, we denote by  $[\eta]_m \in \C$
  the coefficient  of the monomial $t^m$ in $\eta(t)$ and by   $ \Supp(\eta) \subset \Z$
     its {\bf support},  consisting of the 
     exponents $m$ with non-zero coefficients $[\eta]_m$. 
\end{definition}

Therefore, a meromorphic series $\eta\in \C((t))$ may be written as:
   \[   \eta(t) = \sum_{m \in \Supp(\eta)}   [\eta]_m \:   t^m. \]
   
We will also use series with \emph{rational} exponents, but \emph{such that their 
support has bounded denominators}:

\begin{definition}   \label{def:NPseries}
    A {\bf Newton-Puiseux series} $\psi$ 
    in the variable $x$ is a power series of the form 
   $\eta(x^{1/n})$, where $\eta(t) \in \C[[t]]$  and   $n \in \N^*$. 
    For a fixed $n \in \N^*$, they form the ring $\C[[x^{1/n}]]$. 
    Its field of fractions is denoted $\C((x^{1/n}))$.
  \end{definition}
  
  One extends immediately to Newton-Puiseux series $\psi$ the notion of support (which 
  is a set with bounded denominators in the sense of Definition \ref{defnumsem}) 
  and the notation $[\psi]_m$ for their coefficients (where now $m \in \Q_+$). 
 
  Denote by: 
   \[ \C[[ x^{1/ \N} ]]:= \bigcup_{n \in \N^*} \C[[ x^{1/n}]] \]
   the local $\C$-algebra of Newton-Puiseux series in the variable $x$.

 The algebra  $\C[[ x^{1/ \N} ]] $  is endowed with the natural {\bf order} valuation:
    \[\mbox{ord}_x :  \C[[ x^{1/ \N} ]] \longrightarrow \Q_+\cup \{ \infty \}\]
      which associates to each series $\psi = \eta(x^{1/n}) \in \C [[ x^{1/n} ]]$
       the minimum of its support. The {\bf dominating coefficient} of a 
       Newton-Puiseux  series  $\psi$ is  the coefficient of 
       its term of exponent $\mbox{ord}_x(\psi)$.

 The field of fractions of the ring $\C[[ x^{1/ \N} ]]$ of  Newton-Puiseux series is:
   \[   \C(( x^{1/ \N})) := \bigcup_{n \in \N^*} \C((x^{1/n})). \]
 
 One has the following fundamental theorem 
 (see for instance, \cite[Chapter 5.1]{DJP 00},  \cite[Chapter 7]{F 01}, 
 \cite[Chapter 2.1]{C 04} or \cite[Chapter 2]{W 04} 
 for a proof), which explains 
 the reason why we need to work with Newton-Puiseux series even if we are interested 
 primarily in series with integral exponents:
 
 \begin{theorem}[{\bf The Newton-Puiseux theorem}]  \label{NPthm}
 $\,$
 
 \noindent
       Any monic reduced polynomial $f \in \C[[x]] [y]$ of degree $n \in \N^*$ 
       has $n$ roots in $ \C[[ x^{1/ \N} ]]$. If $f$ is moreover irreducible, then those roots 
       are precisely the series of the form: 
       \begin{equation}   \label{eq: gal}
       		\psi_\rho := \eta(\rho \cdot x^{1/n}), 
       \end{equation} 
       where $\psi= \eta(x^{1/n}) \in  \C[[ x^{1/ n} ]]$ is any one of them and 
       $\rho \in \C^*$ varies among the multiplicative subgroup $G_n$ of  $(\C^*, \cdot)$ 
       of $n$-th roots of $1 \in \C^*$.
 \end{theorem}

\begin{remark}  \label{Galact}
If $f \in \C[[x]][y]$ is a monic irreducible polynomial of 
degree $n \in \N^*$and if $\psi \in  \C(( x^{1/ n})) $ is a root  of $f$, 
then the field extension:  
\[
\C ((x)) \: \subset \: \C((x))[y] /(f) \:  \simeq  \: \C((x))[\psi]  \: = \: \C(( x^{1/n})) 
\]
is Galois. Its Galois group is isomorphic to the 
group  $G_n$, acting on $\C(( x^{1/n}))$ by: 
\begin{equation}   \label{eq: Galois}
         (\rho, x^{1/n}) \to \rho \cdot x^{1/n}, \mbox{ for all } \rho \in G_n.
 \end{equation}
The series $\psi_\rho$ in Theorem \ref{NPthm} are precisely the conjugates of $\psi$ 
under this action. 
\end{remark}
 
 Given a Newton-Puiseux series $\psi$, there exists an infinite number of choices 
 of $n \in \N^*$ such that $\psi \in \C[[x^{1/n}]]$. This is simply due to the fact 
 that $\C[[x^{1/n}]]  \subset \C[[x^{1/m}]]$ whenever $n$ divides $m$. One may 
 get nevertheless a canonical choice of $n \in \N^*$ by asking it to be \emph{minimal}: 
 
   \begin{definition}  \label{def: prim} 
       If $\psi \in \C (( x^{1/\N} ))$, a representation $\psi = \eta (x^{1/n})$ 
       with $\eta(t) \in \C[[t]]$ and $n \in \N^*$ is called \textbf{primitive} if 
          $n$ is the lowest common denominator of the exponents of $\psi$. 
     \end{definition}
        
     \begin{example} \label{exconj}
         Assume that $\psi = x^{5/2} +  x^{8/3}$. Then $\psi= \eta(x^{1/6})$, with 
         $\eta(t) = t^{15} +   t^{16} $. This defines a primitive representation of $\psi$. 
         Writing now 
       $\psi= \eta_1(x^{1/12})$ with $\eta_1(t) = t^{30} + t^{32}$, 
       one gets a non-primitive representation.
        Let us consider a $6$-th root of unity $\rho \in \C^*$. Then:  
         $$ \psi_\rho = \eta(\rho \: x^{1/6}) = 
             \rho^{15}  x^{5/2} + \rho^{16} x^{8/3} = 
                \rho^{3}  x^{5/2} + \rho^{4} x^{8/3}.$$
     \end{example}
          
Among the exponents of a Newton-Puiseux series, several are distinguished by looking 
at the way they may be written as quotients of coprime integers: 

\begin{definition}  \label{DEFcharexp}
    Let  $\psi \in \C[[ x^{1/ \N} ]]$ be a nonzero Newton-Puiseux series with zero constant term.  
    The set $ \cE(\psi)$ of \textbf{characteristic exponents} of $\psi$ 
    consists of the elements of the 
    support of $\psi$ which, when written as quotients of integers,  need a denominator 
    strictly bigger than the lowest common denominator of the previous exponents. That is:
     \[ 
     \cE(\psi) := \left\{ l \in \Supp(\psi) \:  | \:  N_{l} \cdot l \notin \Z \right\},
   \mbox{ where }  N_{l} := \min \left\{N \in \N^* \: |  \:  
         \left( \Supp(\psi) \cap [0, l) \right) \subset \dfrac{1}{N} \Z   \right\}.
      \]
   The {\bf sequence of characteristic exponents} of $\psi \in \C[[ x^{1/ \N} ]]$ is defined by 
   writing the elements of $ \cE(\psi)$ in increasing order.
\end{definition}

 \begin{example} \label{exchar}
     Both Newton-Puiseux series $x^{5/2} + x^{8/3}$ and 
         $2x - x^{5/2} + x^{8/3} - 3 x^{7/2} + x^{23/6}$
   have  the same set $\{5/2, 8/3\}$  of characteristic exponents.
   \end{example}
   
   \begin{remark} \label{rem:chardef}
   According to Enriques  and Chisini \cite[page 375]{EC 17}, it was Smith \cite{S 73} and 
Halphen \cite{H 76} who discovered in the years 1870 that  special exponents of a 
Newton-Puiseux series are particularly important if one wants to compute the 
intersection number of two plane branches starting from their Newton-Puiseux series. 
This information was repeated by Zariski \cite[Ch. 1]{Z 35}, 
but without citing anymore their precise papers. Those special exponents were 
called \emph{characteristic} by Halphen in \cite[Sect. 1.1]{H 76}, which is also the paper in which 
he stated his inversion theorem for both exponents and coefficients. This denomination 
remained, but with slightly variable meanings (see also Remarks \ref{charseq} and  \ref{remterm}). 
Let us mention that Smith did not name those special exponents (which he defined in 
\cite[Sect. 8]{S 73}). 
  \end{remark}

The set $\cE(\psi)$ of characteristic exponents of $\psi$
is necessarily finite, even if the series has infinite support. 
More precisely, if $\psi \in \C[[ x^{1/ n} ]]$, then $\cE(\psi)$ 
has at most as many elements as the number of factors of the prime 
factorisation of $n$. The set $ \cE(\psi)$ may also be characterized 
using the Galois action, as the set of orders $\mbox{ord}_x(\psi_{\rho} - \psi)$, 
when $\rho$ varies in $G_n \setminus \{ 1 \}$ (see for instance \cite[Prop. 4.13]{W 04}). 
Note that, because all the conjugates $\psi_{\rho}$  have the same support 
(by the explicit description of the Galois action recalled in Remark \ref{Galact}), 
they also have the same set of characteristic exponents, a fact implicitly used in Definition 
\ref{defcharexpbranch} below.

\medskip

One has the following particular case of the \emph{Weierstrass preparation theorem} 
(see for instance \cite[Chap. 3.2]{DJP 00} or \cite[Chap. 6]{F 01}): 

\begin{theorem}  \label{thmunimon}

Let $f \in \C [[ x, y]]$ be a series such that $\mathrm{ord}_y ( f(0,y) ) = n \in \N^*$. 
Then, there exist a unique monic polynomial $ F \in \C [[x]][y]$ of degree $n$ 
and a unique unit $\epsilon \in \C [[x, y]]$ such that: 
\begin{equation} \label{eq:W}
       f = \epsilon \cdot F.
\end{equation} 
In addition, $f$ is irreducible in $\C [[ x, y]]$ if and only if the polynomial $F$ is irreducible 
in $\C [[x]][y]$. 
\end{theorem}

This theorem  allows us to introduce the following vocabulary:

\begin{definition} \label{defcharexpbranch}  
  Let  $f \in \C [[ x, y]]$ be an irreducible series such that $\mathrm{ord}_y f(0,y) = n \in \N^*$. 
  The polynomial $F \in \C [[x]][y]$ provided by Theorem \ref{thmunimon} is called the 
   \textbf{Weierstrass polynomial} associated to the series $f \in \C [[ x, y]]$ relative to $x$.
    Then, the {\bf Newton-Puiseux series of $f$ relative to $x$} are the roots of the associated 
    Weierstrass polynomial $F \in \C[[x]][y]$, in the ring $ \C[[x^{1/\N}]]$.
   If $f$ is irreducible, then its  {\bf characteristic exponents} relative to $x$ are the 
   characteristic exponents of any one of those roots. 
\end{definition}

\begin{remark} \label{charseq}
Let us explain how the previous algebraic notions apply in the geometrical setting of a branch, 
an irreducible  germ  of complex analytic plane curve $C$ 
on a germ $S$ of  smooth complex analytic surface. 
Choose a local system of coordinates $(x,y)$ on $S$. 
Then, the branch $C$ is defined by an irreducible series $f  \in \mathbb{C} [[ x, y ]]$.
If the $y$-axis is transversal to the tangent line of the branch $C$, then one may show 
that the characteristic exponents 
of $f$ relative to $x$ are always the same (see for instance 
\cite[Thm. 3.5.6]{W 04}). One speaks in this case about {\em generic 
characteristic exponents}. The notations used for them by Zariski in 
\cite[Ch. 1]{Z 35}, \cite[Sect. 3]{Z 68} 
and \cite[Sect. II.3]{Z 86} are common nowadays: 
$\left(\frac{m_1}{n_1}, \frac{m_2}{n_1 n_2}, \dots , \frac{m_g}{n_1 \cdots n_g}\right)$. 
At least since \cite[Sect. 3]{Z 68},  Zariski uses also a 
\emph{characteristic sequence} $(\beta_0, \beta_1, ..., \beta_g)$ of natural 
numbers instead of the sequence of generic characteristic exponents, 
which may then be reconstructed as  $\left(\frac{\beta_1}{\beta_0}, ..., 
\frac{\beta_g}{\beta_0} \right)$ (here   $ \beta_0$ is the multiplicity of 
the branch $C$, that is, the minimal degree of the monomials of $f(x,y)$). 
We do not use the previous notations in this paper for two reasons: on one side 
we never need a genericity hypothesis on the coordinate system relative to $C$ 
and on the other side we find the related notion of essential exponent relative to $1$ 
(see Definition \ref{essell}) better suited to a simple formulation of the 
Halphen-Stolz inversion theorem (see Remark \ref{compcharess}). 
\end{remark}

\section{A calculus for the irreducible terms of invertible entire series}
\label{sec:irred}

In this section we introduce several notions allowing to study the supports of 
Newton-Puiseux series and the semigroups generated by them: their 
\emph{irreducible elements} (see Definition \ref{defindec}) 
and their \emph{essential elements} (see Definition \ref{essell}) relative to 
an arbitrary natural number. We concentrate then on the entire series with non-zero  
constant terms.  If $\phi$ is such a series, we introduce its \emph{dual} $\widecheck{\phi}$ 
 and we show that the coefficients of the monomials  with \emph{irreducible} exponents 
 in the positive integral powers of $\phi$ and in the dual 
 $\widecheck{\phi}$ may be deduced from those of $\phi$ 
 by simple formulae (see Proposition \ref{esspowers}). 
 In Section \ref{sect:applinv} we will apply those formulae to the \emph{essential} 
 exponents relative to well-chosen natural numbers, in order to prove the 
 Halphen-Stolz inversion theorem. 
  \medskip

Next definition introduces vocabulary about the sets of rational numbers
which may appear as supports of Newton-Puiseux series:
  
  \begin{definition} \label{defnumsem}
      A {\bf set with bounded denominators} is  a non-empty (possibly infinite) 
      set $E \subset \Q$ such that there exists $n \in \N^*$ with $E \subset \dfrac{1}{n} \N$. 
     We denote 
      by $\N^* \: E \subset \Q_+$ the {\bf semigroup generated by $E$}, 
     that is, the set of \emph{non-empty} finite 
     sums of elements of $E$. Analogously, we denote by $\Z \: E \subset \Q$ 
     the {\bf group  generated by $E$}.  
  \end{definition}
  
  Note that the semigroup $\N^* \: E$ contains $0$ (that is, it is a monoid for addition)
  if and only if $E$ does.

  Given a set with bounded denominators, we will be interested 
  in its \emph{irreducible elements}:
     
  \begin{definition} \label{defindec}   
     Assume that $E \: \subset \Q_+$ is a set with bounded denominators. 
     We denote by $\Irr(E)$ the set of {\bf irreducible elements} of 
     $E$, that is, the subset of $E$ formed by those elements which cannot 
     be written as sums of at least two elements of $E\setminus \{0\}$. The elements 
     of $E$ which are not irreducible are called {\bf reducible}. \index{reducible elements}
      If $E$ is the support $\Supp(\psi)$ of a Newton-Puiseux series $\psi$, then we 
        write also $\Irr(\psi) := \Irr(\Supp(\psi))$, and we call it the {\bf set of irreducible  
        exponents of $\psi$}.  
  \end{definition}
  
  \begin{remark} \label{minirr} 
     Let $E  \subset \Q_+$ be a set with bounded denominators.
     Notice that if 
     $E$ contains $0$, then $0$ is by definition an irreducible element of $E$. 
     More generally, the minimum of $E$ is always irreducible in $E$. 
  \end{remark}
  
  \begin{example}  \label{exirred}
      Assume that $E = \{6, 15, 16, 21, 23 \}$. Then  $ \Irr(E) = \{ 6 , 15, 16, 23 \}$.
       Note that $21$ is reducible in $E$, because it is equal to the sum $6 +  15$ and $6, 15 \in E$. 
  \end{example}
  
  The sets of irreducible elements of $E$ and of the semigroup it generates coincide: 

\begin{lemma} \label{twoind} 
    Assume that $E \subset \Q_+$ 
    is a set with bounded denominators. Then: 
      \begin{enumerate} 
           \item $\Irr(E) = \Irr(\N^*\: E)$ and this set is the minimal 
    generating set of the semigroup $\N^* \: E$, relative to the inclusion partial order 
    on the set of its generating sets. 
           \item The set $\Irr(E)$ is finite. 
       \end{enumerate}
\end{lemma}  
  
  \begin{proof}
      Multiplying $E$ by a convenient rational number, we may 
      restrict to the sets $E \subset \N$ whose elements are globally coprime. 
      
      \begin{enumerate}
          \item Both inclusions between $\Irr(E)$ and $\Irr(\N^*\: E)$ are immediate 
       to check, therefore we will assume from now on that the two sets are equal. 
       
      Let us prove the minimality property of $\Irr(\N^*\: E)$. 
      Consider another generating set $A$ of $\N^*\: E$ 
       and $a \in \Irr(\N^*\: E)$. As $A$ is generating, $a$ may be written 
       as a sum of elements of $A$. 
       If this sum were non-trivial, then $a$ would not be irreducible in $\N^*\: E$. Therefore 
       $a \in A$, which shows the desired inclusion $\Irr(\N^*\: E) \subset A$. 
       
          \item In order to show that $\Irr(E)$ is finite, it is enough to show that 
       $ \Irr(\N^*\: E)$ is finite. The semigroup $\N^*\: E$ being generated by globally 
       coprime elements, it has \emph{finite conductor}, that is, there exists 
       $c \in \N$ such that all natural numbers greater than or equal to $c$ belong to $ \N^*\: E$ 
       (see \cite[page 82]{W 04}; in this case, the smallest such $c$ is called the \emph{conductor} 
       of the numerical semigroup $\N^*\: E$). 
       But this implies that $\Irr(\N^*\: E) \subset \{0, 1, ..., 2c-1 \}$. Indeed, 
       any element $l \geq 2c$ of the semigroup may be written in the form 
       $c  + d$, with  $d \geq c$, that is, as a non-trivial sum 
       of elements of the semigroup.   
      \end{enumerate}  
  \end{proof}

  We will be especially interested in particular sequences of irreducible elements of a 
  given set $E$ with bounded denominators:
  
  \begin{definition}  \label{essell} 
    Let us consider a set $E \subset \Q_+$ with bounded denominators and 
     an integer $p \in \N^*$. Then the {\bf sequence 
     $\Ess(E,p) := \left(\Ess(E,p)_l \right)_{l}$
     of essential elements of $E$ relative to $p$}  is defined 
     inductively by:
         \begin{itemize}
             \item $\Ess(E,p)_0 := \min E$. 
             \item If $l \geq 1$, then the term $\Ess(E,p)_l$ is defined if and only if 
                 $ E$ is not included in the group \linebreak 
                     $\Z\{p, \Ess(E,p)_0, ...,  \Ess(E,p)_{l -1} \}$. In this case: 
               \[  \Ess(E,p)_l := \min \left(E \setminus \Z\{p, \Ess(E,p)_0, \dots, \Ess(E,p)_{l -1}  \} \right).  \] 
         \end{itemize}
    \end{definition}
  
  The following basic property of this notion is a direct consequence of the definition:
  
  \begin{lemma}  \label{ess-P}
        Assume that $E \subset \Q_+$ has bounded denominators and take $p,q \in \N^*$. Then:
              $$ q \:  \Ess(E,p) = \Ess \left( q  E,  qp \right).$$
  \end{lemma}
  
  One has also: 
  
  \begin{lemma}  \label{finess}
       Assume that $E \subset \Q_+$ has bounded denominators and that $p \in \N^*$. 
       Then the sequence of essential exponents $ \Ess(E,p)$ is finite. 
  \end{lemma}
  
  \begin{proof}
         Lemma \ref{ess-P} implies that it is enough to consider the case where 
         $E \subset \Z$. Then one has by definition the strict inclusions: 
              $$\Z\{p, \Ess(E,p)_0, \dots, \Ess(E,p)_{l  -1}  \}  \subsetneq 
                    \Z\{p, \Ess(E,p)_0, \dots, \Ess(E,p)_{l}  \}, \:  \mbox{ for all } l \geq 1 $$
          (for which the term $\Ess(E,p)_l $ is defined). Any ascending chain of subgroups 
           of $\Z$ being stationary, we deduce that the sequence of essential exponents 
           is finite. 
  \end{proof}

  \begin{example}  \label{exess}
     Let us consider again the set $E = \{6, 15, 16, 21, 23 \}$ from Example \ref{exirred}. Here are 
     its sequences of essential elements relative to the numbers $p \in \{1, ..., 12 \}$:
         $$ \left\{  \begin{array}{l}
                              \Ess(E, 1) = \Ess(E,5) = \Ess(E, 7) = \Ess(E, 11) = (6),  \\
                              \Ess(E,2) = \Ess(E,4) = \Ess(E, 8) = \Ess(E, 10) = (6, 15),   \\
                              \Ess(E,3) = \Ess(E, 9) = (6, 16),  \\
                              \Ess(E,6) = \Ess(E, 12) = (6, 15, 16). 
                          \end{array}  \right. $$
       Lemma \ref{ess-P} implies that: 
           $      \Ess \left( \left\{1, \frac{5}{2}, \frac{8}{3}, \frac{7}{2}, \frac{23}{6} \right\},1\right) 
                     = \frac{1}{6} \Ess(E, 6) = \left(1, \frac{5}{2}, \frac{8}{3} \right). $
       The set $\left\{1, \frac{5}{2}, \frac{8}{3}, \frac{7}{2}, \frac{23}{6} \right\}$ is 
       precisely the support of the second series considered in Example 
       \ref{exchar}, whose sequence of characteristic exponents 
       is $\left(\frac{5}{2}, \frac{8}{3} \right)$. 
       Note that its sequence of essential exponents relative to $1$ 
       may be obtained from the characteristic sequence by adjoining to it as initial term 
       the order of this series (which is in this case equal to $1$). 
       We will see in Lemma \ref{charess} that this is a general fact. 
  \end{example}

  The following lemma shows that the non-zero essential elements of a set $E$ 
  relative to any positive integer are irreducible elements of $E$:
  
  \begin{lemma}  \label{essind} 
     The essential elements of a set $E \subset \Q_+$ with bounded denominators 
     relative to a number $p \in \N^*$ are irreducible elements of $E$. 
  \end{lemma}
  
  \begin{proof} 
     If $\Ess(E,p) = \left( \epsilon_0, \dots, \epsilon_d \right) $, then  we have that 
  $\epsilon_0 = \min E$, which is an irreducible element of $E$. 
  
 \noindent Let us show that the property is also true for 
     $\epsilon_l$, where $l \geq 1$. If $\epsilon_l$ was reducible, 
     then it would be a non-trivial sum of elements of $E$, which would therefore be strictly 
     less than $\epsilon_l$. By the definition of $\epsilon_l$, the terms of this sum  
     would belong to the  group  $\Z\{p, \epsilon_0, \dots, \epsilon_{l-1} \}$.  
     This would imply that  $\epsilon_l$ belongs also to this group, which 
     contradicts Definition \ref{essell}. 
  \end{proof}

\begin{lemma} \label{eqess1}

Assume that $E \subset \Q_+$ has bounded denominators and that $p \in \N^*$. 
       Then we have the following equality of essential sequences:
       \[
        \Ess (E,p) = \Ess ( \Irr(E), p).
        \]
\end{lemma}
\begin{proof}  

If $\Ess (E, p) = (\epsilon_0, \dots, \epsilon_d)$ and 
$\Ess( \Irr(E), p) = (\epsilon_0', \dots, \epsilon_{d'}')$, 
then Remark \ref{minirr} and Definition \ref{essell} imply that:  
\[
\epsilon_0 = \min E = \min \Irr (E) = \epsilon_0'.
\]
Assume by induction that $\epsilon_0 = \epsilon_0', \dots, \epsilon_{l-1} =\epsilon_{l-1}'$, 
for $1 \leq l <d$. Then we get:  
\[ \epsilon_{l}' =  \min (\Irr(E) \setminus \Z \{ p, \epsilon_0, \dots, \epsilon_{l-1} \}) \leq \epsilon_l =  
\min (E \setminus \Z \{ p, \epsilon_0, \dots, \epsilon_{l-1} \}), \]
since $\epsilon_l$ is an irreducible element of $E$  by Lemma \ref{essind}. 
The inclusion $\Irr(E) \subset E$ implies also that $\epsilon_l  \leq \epsilon_{l}'$, hence 
$\epsilon_l = \epsilon_l'$. This proves that $\epsilon_l = \epsilon_l'$ for $0 \leq l \leq d$ and 
$d\leq d'$. By Lemma \ref{essind} 
and the definition of the essential exponents,  one has the inclusions
$   \Irr (E) \subset E \subset \Z \{ p, \epsilon_0, \dots, \epsilon_{d} \}$, which imply that $d' = d$.
\end{proof}

  \begin{definition} \label{essell-2}
            If  $\psi  \in \C[[x^{1/n}]] $ is  a non-zero series and $p \in \N^*$, then we will 
        write: 
            \[\Ess(\psi, p) := \Ess(\Supp(\psi), p),\]
          and we will speak about the 
        {\bf sequence of essential exponents of $\psi$ relative to $p$}.  
  \end{definition}
  
 The characteristic exponents of a Newton-Puiseux series   
 are intimately related to its essential exponents relative to $1$: 
    
   \begin{lemma}  \label{charess}
        Let $(\alpha_1,  \dots,  \alpha_g)$ be the sequence of characteristic 
        exponents of  a series $\psi \in \C[[ x^{1/n}]]$. It may be obtained from  
        the sequence $(\epsilon_0, \epsilon_1, ..., \epsilon_d)$ of essential exponents 
        of $\psi$ \emph{relative to $1$} in the following way: 
           \begin{itemize}
               \item If $\epsilon_0 \notin \Z$, then $g = d + 1$ and 
                    $\alpha_i = \epsilon_{i-1}$ for all 
                    $i \in \{1, ..., d+1\}$.
                \item If $\epsilon_0 \in \Z$, then $g = d$ and $\alpha_i = \epsilon_i $ for all 
                    $i \in \{1, ..., d\}$.
           \end{itemize}
    \end{lemma} 
    
    \begin{proof}   $\: $
    
       $\bullet$ \emph{Consider first the case in which $\epsilon_0 \notin \Z$}. 
        As the first characteristic exponent 
          is the minimal non-integral exponent in the support of $\psi$, we deduce that 
          $\alpha_1 = \epsilon_0$. 
          Assume by induction that $\alpha_i = \epsilon_{i -1}$ for all $i \in \{0, ..., l\}$. 
          Definition \ref{essell} implies that  
           $\epsilon_l $ is the first exponent of $\Supp(\psi)$ which 
           is strictly greater than $\epsilon_{l-1}$ and which cannot be written as a fraction 
           whose denominator is the least common denominator of the previous 
           exponents in the support of $\psi$.
        By Definition \ref{DEFcharexp}, we get that  $\alpha_{l+1}= \epsilon_{l}$. 
                     
           \medskip
           
    $\bullet$ \emph{Consider now the case in which $\epsilon_0 \in \Z$.} 
   By definition, $\epsilon_0 \in \N$ cannot be a characteristic exponent of 
        $\psi$ and  $\alpha_1 = \epsilon_{1}$. The result follows by induction, 
         using the same argument as in the first case.    
        \end{proof}

  In the rest of this section we will be especially interested in entire series
  with non-zero constant term, that is, in invertible elements of the multiplicative  
  monoid $(\C[[t]], \cdot)$. They form a multiplicative group, which we will denote by
  $(\C[[t]]^*, \cdot)$.

 Note that the entire series of the form $t \: \phi(t)$ for $\phi \in \C[[t]]^*$, that is, the 
 entire series of order $1$, form the 
  (non-commutative) group under \emph{composition} of series which admit a 
  \emph{reciprocal} (an inverse for composition). We denote by  $(t \ \C[[t]]^*, \circ)$
  this group. Division by $t$ transforms it bijectively into $(\C[[t]]^*, \cdot)$, but  is not a 
  morphism of groups. What is essential for us is that 
  \emph{the inversion for composition becomes an 
  involution of the set $\C[[t]]^*$ which has special properties with 
  respect to the terms whose exponents are irreducible} (see Proposition \ref{esspowers} 
  (\ref{essdual})). We use the following vocabulary for this involution:

  \begin{definition} \label{dualser}
      If $\phi \in \C[[t]]^*$, then its {\bf dual} \index{dual!of a series} 
      is the unique entire series 
      $\widecheck{\phi} \in \C[[u]]^*$ such that $u \: \widecheck{\phi}(u)$ and 
      $t \: \phi(t)$ are reciprocal. 
  \end{definition}

\begin{remark} 
If $\phi \in \C[[t]]^*$, then setting $u= t  \: \phi(t)$ defines a \textit{change of variable} in the ring 
$\C[[t]]$. Notice that  $\C[[t]] = \C[[u]]$ and  by Definition \ref{dualser} the following equivalence holds:
\begin{equation} \label{f:eq}
u= t  \: \phi(t) \Leftrightarrow t = u \: \widecheck{\phi} (u).
\end{equation}
\end{remark}

We use two variables $t$ and $u$ in our notations for a dual pair of series, in order 
to relate them easily from the notational point of view  to the two sets of 
Newton-Puiseux series of an irreducible $f(x,y) \in \C[[x,y]]$, which depend 
on the two variables $x$ and $y$ (see Section \ref{sect:applinv}). 

\medskip

The following proposition expresses the coefficients of the positive integral powers and of the dual 
of an entire series $\phi \in \C[[t]]^*$  in terms of those of $\phi$. 
It may be deduced from the statement and the proof of Wall's \cite[Lemma 3.5.4]{W 04}. 
But as we could not find it formulated in the literature and as it lies at the core of our first proof of 
the Halphen-Stolz theorem, we give a detailed proof of it.

 \begin{proposition} \label{esspowers}
          Let $\phi \in \C[[t]]^*$ and $N \in \N^*$. Then: 
            \begin{enumerate}
               \item \label{esspow} 
                    $\Irr(\phi^N) = \Irr(\phi)$. Moreover $[\phi^N]_0 = [\phi]_0^N$ and    \,  
                        $[\phi^N]_r = N\:  [\phi]_0^{N-1} [\phi]_r$, for all \,    $r \in \Irr(\phi) \setminus \{0\}$.
               \item \label{essdual} 
                      $\Irr(\widecheck{\phi}) = \Irr(\phi)$. Moreover 
                     $[\widecheck{\phi}]_0 = [\phi]_0^{-1}$ \mbox{ and } \,   
                     $[\widecheck{\phi}]_r = -  [\phi]_0^{-r -2}  [\phi]_r$, 
                          for all \,    $r \in \Irr(\phi) \setminus \{0\}$.
            \end{enumerate}
  \end{proposition}

  \begin{proof}
     One has:
         \begin{equation} \label{explicite} 
              \phi (t)= \sum_{j \in \Supp(\phi)} [\phi]_j \:  t^j. 
         \end{equation}
      The hypothesis $\phi \in \C[[t]]^*$ translates into $0 \in \Supp(\phi)$, 
      that is, $[\phi]_0 \neq 0$. 
       
       \medskip
       \noindent
          (\ref{esspow}) \emph{Consider first the case of $\phi^N$}.  
      By equation (\ref{explicite}), 
          we have:
             \begin{equation} \label{explicitepow} 
                   \phi^N (t)= \sum_{j_{1}, ..., j_{N} \in \Supp(\phi)} [\phi]_{j_{1}} \cdots  [\phi]_{j_{N}} 
                      \:   t^{j_{1} + \cdots + j_{N}}.   
             \end{equation}   
             
          $\bullet$ \emph{Let us show first that $\Irr(\phi) \subset \Irr(\phi^N)$ and that 
              one has the stated equalities between coefficients.} 
             Consider $r \in \Irr(\phi)$. If $r =0$, one has obviously $r \in \Irr(\phi^N)$ and 
              $[\phi^N]_0 = [\phi]_0^N$.  Assume therefore 
             that $r > 0$. The only way to write  $r$ as 
         a sum of $N$ elements of $\Supp(\phi)$, is that one of them be equal to $r$, and the 
         other ones vanish. There are $N$ different positions in the sum for the non-vanishing 
         one, therefore:
              \[  [\phi^N]_r = N\:  [\phi]_0^{N-1} [\phi]_r, \]
          which is the desired formula. 
          
          In particular, $[\phi^N]_r \neq 0$, which shows that $r \in \Supp(\phi^N)$. If $r$ 
          was reducible in $\Supp(\phi^N)$, it could be written as a non-trivial sum of elements 
          of $\Supp(\phi^N)$. By formula (\ref{explicitepow}), it would also be a non-trivial 
          sum of elements of $\Supp(\phi)$, which would contradict the fact that it is an 
          irreducible element of $\Supp(\phi)$. Therefore $r \in \Irr(\phi^N)$. 
          
          \medskip
          $\bullet$ \emph{Let us show now the reverse inclusion $\Irr(\phi^N) \subset \Irr(\phi)$.}  
  Consider an element $r \in \Irr(\phi^N)$. By formula (\ref{explicitepow}), we know that 
               it may be written as a sum of $N$ elements of $\Supp(\phi)$. In particular, it may 
               be  written as a sum of irreducible elements of $\Supp(\phi)$. By the previous 
               point, we know that those elements are also irreducible in $\Supp(\phi^N)$. 
               Our hypothesis $r \in \Irr(\phi^N)$ implies that there is only one non-zero term in 
               this sum, which proves the desired membership $r \in \Irr(\phi)$. 
       \medskip

         \noindent
         (\ref{essdual})  \emph{Consider now the case of $\widecheck{\phi}$}. 
          Write the analogue
          of (\ref{explicite}) for $\widecheck{\phi}$:
              \begin{equation} \label{explicitedual} 
                    \widecheck{\phi}(u) = \sum_{k \in \Supp(\widecheck{\phi})} [\widecheck{\phi}]_k \:  u^k. 
              \end{equation}
         As $t \: \phi(t) \in \C [[t]]$ and $u \:  \widecheck{\phi}(u) \in \C [[u]]$ are reciprocal series, 
         one has by definition the identity: 
            \[t = (t \: \phi(t)) \:  \widecheck{\phi}( t\:  \phi(t) ), \]
        which, after division by $t$ and combination with the expansion (\ref{explicitedual}), gives:
          \begin{equation} \label{onexp} 
                 1 =  \sum_{k \in \Supp(\widecheck{\phi})} [\widecheck{\phi}]_k \:  t^k \phi(t)^{k+1}.   
          \end{equation}
         Expand now the powers $ \phi(t)^{k+1}$ using equation (\ref{explicitepow}). We get:
           \[ 1 = \sum_{\small{\begin{array}{c}
                                      k \in \Supp(\widecheck{\phi})\\
                                      j_{1}, ..., j_{k+1} \in \Supp(\phi)
                                  \end{array} }}    
                          [\widecheck{\phi}]_k  \:  [\phi]_{j_{1}} \cdots [\phi]_{j_{k+1}} \: 
                             t^{k + j_{1} + \cdots + j_{k+1}}.   \]
          Therefore:
           \begin{equation}  \label{vanish}
                \sum_{\small{\begin{array}{c}
                                      k \in \Supp(\widecheck{\phi})\\
                                      j_{1}, ..., j_{k+1} \in \Supp(\phi) \\
                                      k + j_{1} + \cdots + j_{k+1} = p
                                  \end{array} }}    
                          [\widecheck{\phi}]_k  \:  [\phi]_{j_{1}} \cdots [\phi]_{j_{k+1}}    =0 , \:  
                             \hbox{\rm for all } \: p \in \N^*.
           \end{equation}

           \medskip
           $\bullet$ \emph{Let us show first that $\Irr(\phi) =  \Irr(\widecheck{\phi})$}. 
            By Lemma \ref{twoind} (1), the irreducible elements of a set are determined 
           by the semigroup it generates. In order to show that the sets $\Irr(\phi)$ and 
           $\Irr(\widecheck{\phi})$ coincide, \emph{it is therefore enough to prove that:}          
              \begin{equation} \label{equalsg} 
                  \N^* (\Supp(\phi)) = \N^* (\Supp(\widecheck{\phi})).
               \end{equation} 
           
           The situation being symmetric between $\phi$ and $\widecheck{\phi}$, we may 
           prove only the inclusion: 
              \begin{equation} \label{inclsem} 
                     \N^* ( \Supp(\widecheck{\phi})) \: \subset \:  \N^* (\Supp(\phi)).
             \end{equation}
           
           We will argue by contradiction, assuming that the previous inclusion is false. 
           Consider then: 
                \begin{equation} \label{contrhyp}    
                        r \in  \N^* ( \Supp(\widecheck{\phi})) \: \setminus \:  \N^* (\Supp(\phi)),
               \end{equation}
           which is \emph{minimal} with this property. As $ 0 \in \Supp(\phi)$ , we have $r > 0$. 
                           
             Apply equation (\ref{vanish}) to $p =r$. Consider a tuple: 
                \begin{equation} \label{belonging}
                   (k,  j_{1}, ..., j_{k+1}) \in \Supp(\widecheck{\phi}) \times \Supp(\phi)^{k+1} 
                \end{equation}
           such that: 
                \begin{equation} \label{suml}
                        k + j_{1} + \cdots + j_{k+1} = r. 
                \end{equation}
            Let us show that this implies the equality $k=r$.  Reasoning again by contradiction, 
            assume that $k <r$. As 
            $k \in \Supp(\widecheck{\phi}) \subset \N^* (  \Supp(\widecheck{\phi}))$,  
             the minimality of $r$ shows that $k \in \N^* (  \Supp(\phi))$. 
             Combining condition (\ref{belonging})
           and equation (\ref{suml}), we deduce that $r \in \N^* (  \Supp(\phi))$, which contradicts 
           the assumption (\ref{contrhyp}). 
           
           Therefore, if both (\ref{belonging}) and (\ref{suml}) are true, then $k =r$, which implies 
           that $j_{1} = \cdots = j_{r+1} = 0$. 
           Hence there is only one term in the sum of the left-hand side 
           of equation  (\ref{vanish}) for $p =r$, and we get: 
               \[  [\widecheck{\phi}]_{r} \:  [\phi]_0^{r+1} =0, \]
            which contradicts the assumption that both coefficients $[\widecheck{\phi}]_{r}$ 
            and $[\phi]_0$ are non-zero (as they are associated to elements of the supports 
            of $\phi$ and $\widecheck{\phi}$). 
            
            Our proof of the inclusion (\ref{inclsem})  is finished. Therefore, as explained 
            above, we get the desired equality $\Irr(\widecheck{\phi}) = \Irr(\phi)$.
            
                       \medskip
                       
           $\bullet$ \emph{Let us prove the identities relating the coefficients 
              associated to the irreducible exponents of $\phi$ and $\widecheck{\phi}$.}
            Consider $r \in \Irr(\phi) = \Irr(\widecheck{\phi})$. Look again at the tuples 
         $ (k,  j_{1}, ..., j_{k+1})$ satisfying the conditions (\ref{belonging}) and (\ref{suml}) above. 
         
         If $k \notin \{0, r\}$, then  at least one of the numbers $j_{1}, ..., j_{k+1}$ would not 
         vanish. Equation (\ref{suml}) gives a non-trivial 
         decomposition of $r$ inside 
         $\Supp(\phi)  \cup \Supp(\widecheck{\phi}) \subset  \N^*( \Supp(\phi)) 
         \stackrel{ (\ref{equalsg}) }{=}  \N^*( \Supp(\widecheck{\phi}))$,  
          which shows that $r \notin \Irr(\N^* \: \Supp(\phi))$. 
         This contradicts Lemma \ref{twoind} (1). 
         
         Therefore, one has necessarily $k=0$ or $k = r$. Both possibilities determine 
         completely $( j_{1}, ..., j_{k+1})$ through equation (\ref{suml}). Applying equation 
         (\ref{vanish}) to $p =r$, we get: 
            \begin{equation} \label{eqcoefsym} 
                  [\widecheck{\phi}]_0 \:  [\phi]_r +  [\widecheck{\phi}]_r \:   [\phi]_0^{r+1} =0.  
            \end{equation}
       The equalities (\ref{eqcoefsym}) and $[\widecheck{\phi}]_0 = [\phi]_0^{-1}$ 
          imply the formula 
          for $[\widecheck{\phi}]_r$ written in the statement of the proposition.       
  \end{proof}

  Combining Proposition \ref{esspowers} with Lemma  \ref{eqess1}, we get:

  \begin{corollary} \label{coress}
       Let $\phi \in \C[[t]]^*$ and $N, p \in \N^*$. Then the sequences of essential exponents 
       of $\phi, \phi^N$ and $\widecheck{\phi}$ relative to $p$ coincide. 
  \end{corollary}

In the next section we apply Proposition \ref{esspowers}  and its Corollary \ref{coress}
  in  order to relate the 
\emph{essential exponents relative to $1$} and their \emph{coefficients} 
for the Newton-Puiseux series of an irreducible series $f(x,y) \in \C [[x,y]]$.

\section{Applications to inversion formulae for Newton-Puiseux series}
\label{sect:applinv}

 Let $f(x,y)\in \C[[x,y]]$ be an irreducible formal power series. 
 One has therefore associated Newton-Puiseux series relative to both coordinates $x$ and $y$. 
In this section we prove in two ways 
the \emph{Halphen-Stolz theorem} (Corollary \ref{invPuiseux}), which relates the coefficients 
of the terms with essential exponents relative to $1$ in both series. The first proof, 
summarized in the flow-chart (\ref{reasoning}),
applies directly the results of the previous section. The second one passes through a 
more general result, allowing to compute recursively \emph{all} the coefficients of one series 
in terms of those of the other one (see Proposition \ref{invexpl}). In turn, this proposition 
is a consequence of a version of the classical Lagrange inversion formula 
(see Theorem \ref{Lagrinv}).

\medskip

\subsection{The first proof of the Halphen-Stolz theorem} 
\label{sect:firstproof}
$\:$ 
\medskip

There is no natural bijection between the Newton-Puiseux series of a formal power series 
$f(x,y)$ relative to $x$ and $y$,  
for the simple reason that their numbers are in general different. We want to explain 
first that if one takes adequate roots of them, then one gets two sets  which are 
naturally in a bijective correspondence (see  Proposition \ref{propbij}). 

 Let us denote by  $\eta (x^{1/n}) \in  \C[[x^{1/n}]]$ a Newton-Puiseux series of $f(x,y)$  
 with respect to $x$, where $\eta \in \C[[ t]]$. 
 We assume that the representation of this Newton-Puiseux power series is primitive  
(see Definition \ref{def: prim}). The series  $\eta$  is of the form:
     $$\eta =  a \cdot t^{m} + \mbox{higher order terms},$$
  with $m >0$ and $a \in \C^*$. 
 Let us choose an $m$-th root $\tilde{a}\in \C^*$ of $a$.
 Then, we have a unique $m$-th root $t \: \tilde{\eta}(t) \in \C[[t]]$ of $\eta$:
    \begin{equation}  \label{power}
        \eta(t) = (t \:  \tilde{\eta}(t))^m, 
    \end{equation}
 such that the series $\tilde{\eta}$ has constant term $[\tilde{\eta}]_0 = \tilde{a} \ne 0$.

 \begin{example}  \label{ex:varnot}
     Start from a branch $C$ with Newton-Puiseux series:
       $\psi=  x^{3/2} + c \cdot x^{7/4}$,
     where $c \in \C^*$. Therefore, we get $a =1$, $n =4$ and $\psi = \eta( x^{1/4})$, where
        ${\eta}(t) = t^6 + c \cdot t^7 = t^6(1 + c \cdot t).$
     This shows that $m = 6$ and if $\tilde{a} = 1$ then 
         $\tilde{\eta}(t) = (1 + c\cdot t)^{1/6} $. By (\ref{binom}) below
         we have the expansion $\tilde{\eta}(t) = 
          1 + \sum_{k \in \N^*} \binom{1/6}{k} c^k \cdot t^k$.        
 \end{example}
 
 Let us come back to the general case. 
 
 \medskip

 Denote by $\tilde{\xi}(u) \in \C[[u]]^*$ the dual series of $\tilde{\eta}(t)$ (see Definition \ref{dualser}).  
Hence, one has the following equivalence (see (\ref{f:eq})): 
    \begin{equation} \label{invser}
           u = t \:  \tilde{\eta}(t) \:  \Leftrightarrow \:   t = u \: \tilde{\xi}(u). 
    \end{equation}  
    
    \medskip
    
As $\eta(x^{1/n})$ is a Newton-Puiseux-series of $f(x,y)$ relative to $x$, we have:
   \[f(x, \eta(x^{1/n})) =0. \]
 Replacing $x$ by $t^n$ and using the equality (\ref{power}), we get: 
        \begin{equation}   \label{eqsm1}
              f(t^n, (t \: \tilde{\eta}(t))^m) =0. 
         \end{equation} 
By the equivalence (\ref{invser}), we deduce: 
     \begin{equation}  \label{eqsm2} 
            f((u\: \tilde{\xi}(u))^n, u^m) =0. 
     \end{equation}
 Consequently, if one defines: 
    \begin{equation}  \label{powerbis}
        \xi(u)  = (u \:  \tilde{\xi}(u))^n  
    \end{equation} 
 (an equation which is analogous to (\ref{power})), then one sees that:
      \[   f(\xi(y^{1/m}), y) =0, \]
 that is,  $ {\xi}(y^{1/m}) $
  is a Newton-Puiseux series of $f(x, y) $ with respect 
 to the variable $y$. 
 
 In fact, one has the following proposition: 
    
    \begin{proposition}  \label{propbij}
        $\: $ 
       
       \begin{enumerate} 
          \item \label{algbij} 
               The map  $$x^{1/n} \tilde{\eta}(x^{1/n}) \longrightarrow y^{1/m} \tilde{\xi}(y^{1/m})$$ 
              induced by the duality involution on $\C[[t]]^*$, gives a bijection from  the 
              set of $m$-th roots of the Newton-Puiseux series of $f(x,y)$ relative to  $x$ 
              to the set of  $n$-th roots of those relative to  $y$. 
          \item \label{geobij}  
          
                 If $\tilde{f}(t,u) : = f(t^n, u^m) \in \C [[t, u]]$ then: 
          \begin{itemize}
              \item[-]  the series of the form $t \:  \tilde{\eta}(t) \: \in \C [[t]]$ 
                 are the Newton-Puiseux series of $\tilde{f}$ relative to  $t$.
              \item[-] the series of the form $u \: \tilde{\xi}(u) \:  \in \C [[u]]$ 
                 are the Newton-Puiseux series of $\tilde{f}$ relative to $u$.
          \end{itemize}
         \end{enumerate}
         
    \end{proposition}
    
   \begin{proof}  $\: $
     
    (\ref{algbij}) The two sets have both $mn$ elements 
          and the given map is injective 
          because the map $\tilde{\eta} \to \tilde{\xi}$ is an involution. 
          Therefore the given map is bijective. 
        
        \medskip
      
      (\ref{geobij})
       Equation  
           (\ref{eqsm1}) shows that the Newton-Puiseux 
           series of $\tilde{f}$ relative to the variable $t$ are exactly those of the form 
           $t \:  \tilde{\eta}(t)$. 
           The situation is analogous for the series of the form 
           $u \: \tilde{\xi}(u)$. 
   \end{proof}

    The following lemma relates special sequences of essential exponents  
    of the series ${\eta}(t)$ and $\tilde{\eta}(t)$ on one side, and of  the series ${\xi}(u)$ and $\tilde{\xi}(u)$ 
    on another side:

    \begin{lemma} \label{divpower}
        Denote:
           $$     \left \{   \begin{array}{l}
                                    \Ess(\eta, n) = (m , \epsilon_1,  ..., \epsilon_d), \\
                                    \Ess(\xi, m) = (n, \epsilon'_1, ..., \epsilon'_{d'}). 
                             \end{array}   \right. $$ 
        Then: 
           \[   \left \{   \begin{array}{l}
                                    \Ess(\tilde{\eta}, \gcd (n,m)) = 
                                           (0, \epsilon_1 - m, ..., \epsilon_d - m), \\
                                    \Ess(\tilde{\xi}, \gcd (n,m)) = 
                                        (0, \epsilon'_1 - n, ..., \epsilon'_{d'} - n). 
                             \end{array} \right.  \]
    \end{lemma}
    
    \begin{proof}
         By symmetry, it is enough to treat the case of the series $\tilde{\eta}$.
       Since $\tilde{\eta} \in \C[[t]]^*$, 
      Proposition \ref{esspowers} implies that 
     $\Irr( \tilde{\eta}) = \Irr ( \tilde{\eta}^m )$. 
      Then, by Lemma \ref{eqess1},  for any integer $p \in \N^*$ one has: 
   \[
   \Ess ( \tilde{\eta}, p) = \Ess ( \Irr( \tilde{\eta}), p ) =  \Ess (  \Irr ( \tilde{\eta}^m), p ) 
         =  \Ess ( \tilde{\eta}^m, p).
   \]
 Thus,  it is enough to prove that $(0, \epsilon_1 - m, ..., \epsilon_d - m)$ is 
       the sequence of essential exponents 
        of $\tilde{\eta}^m$ relative to $\gcd(n,m)$.
       By formula (\ref{power}), we have that $\tilde{\eta}^m =  t^{-m} {\eta}$, therefore: 
         $ \Supp(\tilde{\eta}^m) = \Supp({\eta}) - m$.
       Using Definition \ref{essell}, we see that we have to prove that:
           \begin{itemize}
             \item $\min (\Supp({\eta}) - m) = 0$. 
             \item For all $k \in \{1, ..., d\}$: 
               $ \epsilon_k - m = \min \left((\Supp({\eta}) - m) \setminus 
                     \Z\{ \, \gcd{(n,m)}, 0,  \epsilon_1 - m, ..., \epsilon_{k-1} - m \} \right)$.   
             \item $\Supp({\eta}) - m \subset 
                   \Z\{\, \gcd{(n,m)}, 0, \epsilon_1 - m, ..., \epsilon_d - m \}$. 
        \end{itemize}
        But all these facts are immediate from the definition of the essential exponents 
            $\epsilon_i$, because:
        \[
              \Z\{ \, \gcd{(n,m)}, 0,  \epsilon_1 - m, ..., \epsilon_{k-1} - m \}  =  
         \Z\{ n , m,  \epsilon_1 , ..., \epsilon_{k-1}  \},
        \]
         for all $1 \leq k \leq d$, an equality which is immediate to check by double inclusion.
    \end{proof}

    We are ready to deduce an \emph{inversion formula}, expressing the sequence of essential 
    exponents of ${\xi}$ relative to $m$ and the associated coefficients 
    in terms of the sequence of essential exponents of ${\eta}$ relative to $n$ 
    and their associated coefficients. We chose to inverse also the order of presentation, 
    by starting from any pair of dual series $(\tilde{\eta}, \tilde{\xi})$ and any pair 
    of positive integers $(m,n)$, and by associating to them the series $(\eta, \xi)$ by the 
    formulae (\ref{power}) and (\ref{powerbis}). In this way, we emphasize only 
    univalued maps, in contrast to their  reciprocals, which involve taking roots.

    \begin{theorem}   \label{inversioncoef}
       Let $\tilde{\eta}\in \C[[t]]^* $ and $\tilde{\xi} \in \C[[u]]^*$ be dual of each other 
       and $m,n \in \N^*$. Let $\tilde{a}$ be the constant term of $\tilde{\eta}$. Denote:
            \[ \left\{ \begin{array}{l} 
                    \eta(t) = (t \: \tilde{\eta}(t))^m, \\
                     \xi(u) = (u\:  \tilde{\xi}(u))^n, 
                \end{array} \right. \]
        and:  
         \[  \left \{   \begin{array}{l}
                                    \Ess(\eta, n) = (m , \epsilon_1, ..., \epsilon_d), \\
                                    \Ess(\xi, m) = (n, \epsilon'_1, ..., \epsilon'_{d'}). 
                             \end{array}   \right.  \]
   Then one has the following inversion formulae for exponents and coefficients: 
   \begin{equation}  \label{eq: inv1}
          d' = d, 
    \end{equation}
  \begin{equation}  \label{eq: inv2}
   \epsilon_k' + m = \epsilon_k + n, \: \hbox{\rm for all } \:  k \in \{1, ..., d \}, 
 \end{equation}
         \begin{equation} \label{eq: inv3}
                  [{\xi}]_{n} =  {\tilde{a}}^{-n} 
            \quad 
            \mbox{ and }  \quad 
              [{\xi}]_{\epsilon_k'} = - \frac{n}{m} \:  {\tilde{a}}^{-n - \epsilon_k} 
              [{\eta}]_{\epsilon_k}, \: \hbox{\rm for all } \:  k \in \{1, ..., d \}. 
          \end{equation}
    \end{theorem}
    
    \begin{proof}
          The entire series $\tilde{\eta}$ and $\tilde{\xi}$ being dual  in the sense 
          of Definition \ref{dualser}, Corollary \ref{coress} shows  that they have the same 
          sequences of essential exponents relative to  $\gcd(n,m)$. Then, Lemma \ref{divpower} 
          allows us to deduce the desired formulae (\ref{eq: inv1}) and (\ref{eq: inv2})  
          relating the two sequences 
          $(\epsilon_k)_k$ and $(\epsilon'_k)_k$.

          \medskip
          Let us pass to the proof of the inversion formula (\ref{eq: inv3}) for the coefficients. 
          
          Equation (\ref{powerbis}) implies that ${\xi} = u^n ({\tilde{\xi}})^n$. Therefore:
              \begin{equation} \label{step1} 
                   [{\xi}]_{\epsilon_k'} = [({\tilde{\xi}})^n]_{\epsilon_k' - n}. 
             \end{equation}
          Combining Proposition \ref{esspowers} and Lemma \ref{essind}, we get: 
             \begin{equation} \label{step2} 
                   [ (\tilde{\xi})^n]_{\epsilon_k' - n} = n \:  [\tilde{\xi}]_0^{n-1} 
                        [\tilde{\xi}]_{\epsilon_k' - n} = n \: \tilde{a}^{-n+1} [\tilde{\xi}]_{\epsilon_k' - n}. 
             \end{equation}
            The same proposition, combined with the equivalent form $\epsilon_k' - n = 
               \epsilon_k - m$ of the equality (\ref{eq: inv2}),  implies that: 
                \begin{equation} \label{step3} 
                        [\tilde{\xi}]_{\epsilon_k' - n} = 
                           - [\tilde{\eta}]_0^{-\epsilon_k + m -2} [\tilde{\eta}]_{  \epsilon_k - m}  = 
                            -\tilde{a}^{-\epsilon_k + m -2} [\tilde{\eta}]_{  \epsilon_k - m} . 
             \end{equation}
             Combining the equalities (\ref{step1}), (\ref{step2}) and (\ref{step3}), we obtain: 
                \begin{equation} \label{step4} 
                     [{\xi}]_{\epsilon_k'} = 
                           - n \:  \tilde{a}^{-\epsilon_k +  m -n -1} [\tilde{\eta}]_{  \epsilon_k - m}. 
             \end{equation}
             Now, from the analogues of equations (\ref{step1}) and (\ref{step2}) for $\tilde{\eta}$, 
             we get:
              \begin{equation} \label{step5} 
                          [\tilde{\eta}]_{  \epsilon_k - m}  
                          = \frac{1}{m} \:  [\tilde{\eta}]_0^{-m +1} [\tilde{\eta}^m]_{\epsilon_k - m}
                          = \frac{1}{m} \:  [\tilde{\eta}]_0^{-m +1} [{\eta}]_{\epsilon_k} 
                          = \frac{1}{m} \:  \tilde{a}^{-m +1} [{\eta}]_{\epsilon_k} . 
             \end{equation}
            Combining formulae (\ref{step4}) and (\ref{step5}), we deduce the
            inversion formula for the coefficients $[{\xi}]_{\epsilon_k'}$, for $ k \in \{1, ..., d \}$. 
     \end{proof}

   Dividing by $n$ all the terms of the sequence $(m, \epsilon_1, ..., \epsilon_d)$, one gets 
    the sequence of essential exponents of $\eta$ relative to $1$ (see Remark  \ref{ess-P}). 
    Similarly, 
    dividing by $m$ all the terms of the sequence $(n, \epsilon'_1, ..., \epsilon'_d)$, one gets 
    the sequence of essential exponents of $\xi$ relative to $1$. Theorem 
    \ref{inversioncoef} translates therefore in the following inversion formula for the 
    Newton-Puiseux series of $f(x,y)$ relative to $x$ and to $y$, 
    which is the theorem of Halphen-Stolz presented in the introduction: 
    
    \begin{corollary}[{\bf The Halphen-Stolz inversion theorem}]   \label{invPuiseux}
        $\,$
        
        \noindent
          Let  $\eta (x^{1/n}) $ and $\xi(y^{1/m}) $ be Newton-Puiseux series of 
          an irreducible formal power series $f(x,y)\in \C[[x,y]]$ relative to $x$ and $y$ 
          respectively. As before, we assume that ${\eta}(t) = (t \: {\tilde{\eta}}(t))^m$  and 
          ${\xi}(u) = (u \: {\tilde{\xi}}(u))^n$, where  
          $\tilde{\eta}(t), \tilde{\xi}(u)$ are dual series and  $[\tilde{\eta}]_0 = \tilde{a}$. 
          Denote:
             \[  \left \{   \begin{array}{l}
                                  \Ess(\eta (x^{1/n}), 1) = \displaystyle{(m/ n, e_1, \dots, e_d)}, \\
                                  \Ess(\xi(y^{1/m}), 1) = (n / m, e_1', \dots, e_{d'}'). 
                           \end{array}   \right.  \]            
          Then one has the following inversion formulae for exponents and coefficients: 
               \begin{equation} \label{eq: invP0}
                 \displaystyle{ d'=d  }. 
           \end{equation} 
           \begin{equation} \label{eq: invP1}
                 \displaystyle{ m (1 + e_k') = {n}(1 + e_k)  } \,  \mbox{  for  all } \,  k \in \{1, ..., d\}. 
           \end{equation} 
           \begin{equation} \label{eq: invP2}
                 [\xi (y^{1/m}) ]_{n/m} = \tilde{a}^{-n} 	
                 \,                   
                 \mbox{ and }
                \,   \,  
                 [\xi (y^{1/m}) ]_{e_k'} = - \dfrac{n}{m} \:  \tilde{a}^{-( 1  + e_k)n} 
                  [\eta (x^{1/n})]_{e_k} \,  \mbox{  for  all } \,  k \in \{1, ..., d\}. 
             \end{equation} 
    \end{corollary}
    
    In the case in which $\tilde{a} = 1$, the inversion formula for 
    the coefficients stated in  Corollary \ref{invPuiseux} may be written 
    in a more symmetric way, easier to remember:
      
    \begin{corollary}  \label{corone}
        Assume moreover that the constant coefficient $\tilde{a}$ of $\tilde{\eta}$ is equal to $1$.
        Then: 
            \[  [\xi (y^{1/m}) ]_{n/m} = 1 = [\eta (x^{1/n}) ]_{m/n} \quad \mbox{ and }  \quad 
     m [\xi (y^{1/m}) ]_{e_k'} + n  [\eta  (x^{1/n}) ]_{e_k} =0  \mbox{ for all } k \in \{1, ..., d\}. \]
    \end{corollary}

  \medskip
 \noindent 
  {\bf Summary of the previous arguments.}
  In order to understand better the line of reasoning we followed till now, 
  the reader may find helpful  the following flow-chart, in which  $f(x,y) \in \C[[x,y]]$ 
  is irreducible: 
    
    \begin{equation} \label{reasoning}
    \xymatrix{
                    & \overset{  \Ess(\eta(x^{1/n}), 1)  }{   \{ \eta(x^{1/n}) \}    }
                    \ar@{<->}[rr]_(.52){(\mathrm{Lm.} \ref{ess-P})}^(.52){x=t^n}      &  & 
                    \overset{ \Ess(\eta(t), n)}{ \{ \eta(t) \}  }    &  & 
                     \overset{ \Ess(\tilde{\eta}(t), \gcd(m,n) )}{    \{ \tilde{\eta}(t) \}  }\ar[ll]^(.55){\begin{psmallmatrix}
                                    ( \mathrm{Prop.} \ref{esspowers} (\ref{esspow}),  \\
                                     \mathrm{Lm.} \ref{divpower})
                                     \end{psmallmatrix}}_(.55){\eta=(t \tilde{\eta})^m} 
                             \ar@{<->}[dd]_{(  \mathrm{Prop.} \ref{esspowers} ( \ref{essdual})  )}^{\begin{array}{c}
                                                  u = t \:  \tilde{\eta}(t)  \\  
                                                  \Updownarrow \\
                                                  t = u \:  \tilde{\xi}(u)
                                             \end{array} }\\
           f(x,y) \; 
           \ar[ur] \ar[dr]        &     &     &   \\
                    & \underset{\Ess(\xi(y^{1/m}), 1)} {\{ \xi(y^{1/m}) \}   }
                          \ar@{<->}[rr]^(.52){(\mathrm{Lm.} \ref{ess-P})}_(.52){y=u^m}      &      &  
                      \underset{ \Ess(\xi(u), m)}{  \{ \xi(u) \} } &   & 
                 \underset{ \Ess(\tilde{\xi}(u), \gcd(m,n) )}{ \{ \tilde{\xi}(u) \} } \ar[ll]_(.55){\begin{psmallmatrix}
                                     (\mathrm{Prop.} \ref{esspowers} (\ref{esspow}),  \\
                                    \mathrm{Lm.} \ref{divpower})
                                     \end{psmallmatrix}}^(.55){\xi=(u \tilde{\xi})^n} 
               }
    \end{equation}
    
Let us explain this diagram: 
   \begin{itemize}
          \item    From the irreducible series $f(x,y) \in \C[[x,y]]$,    
          one gets symmetrically two sets of Newton-Puiseux series 
                    $ \{ \eta(x^{1/n}) \} $ and $ \{ \xi(y^{1/m}) \}$. The first one has 
                    $n$ and the second one $m$ elements. 
           \item Follow now two analogous sequences of transformations of those sets, 
              indicated in the diagram horizontally. We describe them only for the upper 
              line of the diagram. 
           \item  The change of variables $x=t^n$, indicated above the corresponding 
           doubly-arrowed horizontal segment, puts 
              the set $ \{ \eta(x^{1/n}) \} $ in bijection with 
              the set of entire series $\{ \eta(t) \}$.  
            \item Lemma \ref{ess-P}, mentioned below the same  
               arrow, allows to pass from $\Ess(\eta(x^{1/n}), 1)$ to $\Ess(\eta(t), n)$. 
               The corresponding coefficients are unchanged. 
           \item One extracts in all possible ways the $m$-th roots of the series 
                $\eta(t)$. Then  one divides the result by $t$, arriving at a set 
                $\{ \tilde{\eta}(t) \}$ with $mn$ elements. The composition of the 
                two operations is expressed by the formula $\eta=(t \:  \tilde{\eta})^m$, 
                written above the corresponding arrow. 
           \item Combining Proposition \ref{esspowers} (\ref{esspow}) with Lemma 
              \ref{divpower}, one passes from the sequence $\Ess(\eta(t), n)$ to 
              $  \Ess(\tilde{\eta}(t), \gcd(m,n) )  $ and one relates also the 
              corresponding coefficients. 
            \item There is a canonical bijection between the two sets $\{ \tilde{\eta}(t) \}$ 
                and $\{ \tilde{\xi}(u) \}$, indicated by the left vertical double-arrowed segment. 
                This bijection associates two series $\tilde{\eta}(t)$ and $ \tilde{\xi}(u)$ 
                whenever $\tilde{\eta}(t)$ and   $ \tilde{\xi}(u)$ are dual of each other, 
                which may be expressed by the two equivalent equalities marked at 
                the right of the vertical segment. 
            \item Proposition \ref{esspowers} (\ref{essdual}), indicated to the left of 
               the same segment, shows that the two sequences are equal, and allows 
               to relate the corresponding coefficients. Note that this proposition allows 
               in fact to relate the coefficients corresponding to \emph{all} the irreducible  
               exponents of the two dual series, not only those which are essential 
               relative to $\gcd(m,n)$. This is understandable if one thinks that, 
               reading now the diagram from right to left, 
               one may start from \emph{any} pair $(\tilde{\eta}(t), \tilde{\xi}(u))$ 
               of dual series and only \emph{afterwards} choose the pair 
               of positive integers $(m,n)$, independently of the choice of the two dual series. 
               One arrives then at the series $f(x,y)$ by taking either the minimal polynomial 
               of $\eta(x^{1/n})$ or that of $\xi(y^{1/m})$, and multiplying it with an invertible 
               element of the ring $\C[[x,y]]$.
    \end{itemize}

        \begin{remark} \label{compcharess}
        Using Lemma \ref{charess}, the Halphen-Stolz inversion theorem 
        (Corollary \ref{invPuiseux})  may also be expressed 
        in terms of the characteristic exponents of $\eta$ and $\xi$. 
        That lemma shows that the sequences of characteristic exponents 
        of $\eta$ and of $\xi$ do not have necessarily the same lengths, which has as 
        consequence the fact that the elements of the two series which are related 
        have not necessarily the same position in both sequences. For this reason, 
        it is easier to express the inversion formulae as we have done in Theorem 
        \ref{inversioncoef} and in Corollary \ref{invPuiseux}, in terms of the essential exponents. 
        \end{remark}
    
    \begin{remark}  \label{remterm}
       The part of Corollary \ref{invPuiseux} concerning the exponents 
       is usually expressed nowadays 
       in terms of the characteristic exponents and is sometimes attributed to Abhyankar's paper  
       \cite{A 67} of 1967 or to Zariski's paper \cite{Z 68} of 1968. 
       In fact, it was already stated precisely in terms of the sequences of 
       essential exponents relative to $1$ (called there ``\emph{exposants caract\'eristiques}"
       from their second term on) 
       by Halphen \cite[page 91]{H 76} in 1876. But Halphen stated also 
       the previous formulae (of course, with different notations) for the inversion 
       of the corresponding \emph{coefficients}. He did not prove those formulae, and 
       as far as we know, the unique proof was provided by Stolz 
       \cite[page 133]{S 79} in 1879. We searched new proofs 
       because we were not fully convinced by Stolz' arguments and because we wanted 
       to extend the theorem to higher dimensions. 
    \end{remark}

    \medskip
    \subsection{The second proof of the Halphen-Stolz theorem} 
\label{sect:secondproof}
$\:$ 
\medskip
    
    Let us pass now to our second proof of Theorem \ref{inversioncoef}. 
    Corollary \ref{invPuiseux} concerns only the terms of the two Newton-Puiseux series 
    whose exponents are essential relative to $1$. 
    We explain now a way to get formulae for \emph{all} the coefficients of $\xi$ 
    as rational fractions of those of $\eta$. 
    
     \medskip
     
  We recall first a form of {\em Lagrange's inversion formula} which, given two reciprocal 
  entire series  $X$ and $Y$, 
 allows to express the coefficients of the integral powers of $X$  in terms of those of $Y$. 
 Several proofs of it may be found in \cite[Theorem 5.4.2]{S 99}, and historical explanations 
 in \cite[pages 67-68]{S 99}. Let us mention only that the founding result for 
 this kind of formulae was stated by Lagrange in \cite[Par. 16]{L 70}.

  \begin{theorem}[{\bf Lagrange inversion formula}] \label{Lagrinv} 
    $\,$
    
    \noindent
    Let $X(u) \in u \: \C[[u]]^*$ and $Y(t) \in t \: \C[[t]]^*$ be two reciprocal series. 
     For any $p, q \in \Z$, one has: 
       \[p \cdot [X(u)^q]_p = q \cdot [Y(t)^{-p}]_{-q}. \]
  \end{theorem}
  
  Note that the lifting to a negative integral power produces a meromorphic series which 
  has a finite number of terms with negative exponents. 
  
  \medskip
  
 Let us apply Theorem \ref{Lagrinv} in our context.

Recall that $\tilde{a} \in \C^*$ is the  constant term of $\tilde{\eta}$,  hence 
by formula (\ref{power}), we may write: 
    \begin{equation}  \label{factor}
        {\eta}(t) = \tilde{a}^m t^m \left( 1 + \sum_{k > m} c_k t^{k-m} \right).  
    \end{equation}
 Therefore: 
     \begin{equation}  \label{formtilde}
            \tilde{\eta}(t) = \tilde{a} \left(1 + \sum_{k > m} c_k t^{k-m}\right)^{1/m}, 
    \end{equation}
  in which the right-hand-side can be computed using the generalized binomial expansion:
    \begin{equation} \label{binom}
        (1+x)^{r} := 1 + \sum_{k \in \N^*} \binom{r}{k} x^k , \:  \forall \:  r \in \R,
    \end{equation}
 where:
   \begin{equation} \label{bindef}
        \binom{r}{k} = \dfrac{r(r-1) \cdots (r - k+1)}{k !}, \:  \mbox{ for all } r \in \R.
   \end{equation}

Using formula (\ref{formtilde}), one gets the following consequence of Theorem \ref{Lagrinv}, 
which allows to compute the coefficients 
of the Newton-Puiseux series $\xi$ of the irreducible power series $f(x,y)\in \C[[x,y]]$ 
relative to $y$ in terms of those $\eta$ 
relative to $x$. Note that one gets rational fractions whose numerators are polynomials 
with rational coefficients 
in the coefficients $c_k$ and whose denominators are positive integral powers of $\tilde{a}$: 
     
 \begin{proposition} \label{invexpl}
     Assume that:
        \[ \eta (x^{1/n})  = \tilde{a}^m  x^{m/n} \left(  1+ \sum_{k > m}c_k x^{\frac{k-m}{n}} \right).\]
     Then one has the following formula for the coefficients $ [\xi]_{\frac{q}{m}} $  of 
     the corresponding Newton-Puiseux series $\xi(y^{1/m}) \in \C [[y^{1/m}]]$, 
     for all integer $q \geq n$: 
         \[ \displaystyle{  [\xi]_{\frac{q}{m}} = \frac{n}{q} \tilde{a}^{-q} \left[ 1 + \sum_{i \geq 1}
              \binom{-q/m}{i}  \left( \sum_{s> m} c_s x^{\frac{s-m}{n}}\right)^i \right]_{ -1 + \frac{q}{n}}. }\]
 \end{proposition}
 
 \begin{proof}
    By Theorem \ref{Lagrinv} applied after replacing  the pair $(p,q)$ by $(q,n)$,  
    we have:   
    \[ q \cdot [u^n \tilde{\xi}(u)^n]_q  = n \cdot  [ t^{-q} \tilde{\eta}(t)^{-q}]_{-n}.\]
    We get: 
      \begin{equation} \label{Lag-2}
        \begin{array}{lcl} 
             q \cdot [u^n \tilde{\xi}(u)^n]_q 
                 &    \stackrel{(\ref{formtilde})}{=}   
                   & \displaystyle{ n \tilde{a}^{-q} \left[\left( t^{-q}
                       \left(1 + \sum_{s > m} c_s t^{s-m}\right)^{-q/m} \right)
                         \right]_{-n} },
           \end{array} 
         \end{equation}
         and then by (\ref{binom}): 
         \begin{equation} \label{Lag}
           q \cdot [\xi (u) ]_q =  q \cdot [u^n \tilde{\xi}(u)^n]_q  = 
                 n \tilde{a}^{-q} \left[\left( 1 + \sum_{i \geq 1}
              \binom{-q/m}{i}  \left( \sum_{s > m} c_s t^{s-m}\right)^i \right) \right]_{q-n}.
         \end{equation}
       It is enough now to divide by $q$, to replace $u$ by $y^{1/m}$ and $t$ by $x^{1/n}$ in 
       order to get the desired formula. 
 \end{proof}

\medskip 

As a corollary, we obtain:

\medskip

\noindent
\textbf{Second proof of Theorem \ref{inversioncoef}.}
Recall the notation $\Ess(\eta, n) = (m, \epsilon_1, \dots, \epsilon_d)$.
We set then:
\[
\hat{\epsilon}_0 = {n}, \quad  \hat{\epsilon}_1= 
      \epsilon_1 -m +n, \quad  \dots,  \quad \hat{\epsilon}_{d} = \epsilon_d - m+ n.
\] 
We prove  first by induction on the integer $q \geq n$ that if 
\begin{equation} \label{eq:k}
\hat{\epsilon}_{k} \leq q < \hat{\epsilon}_{k+1} \mbox{ for some } k \in \{0, \dots, dÊ\}, 
\end{equation}
then the terms of the sequence $\Ess(\xi, m)$ 
which are lower than or equal to $q$ are precisely 
$\hat{\epsilon}_0 = {n}, \hat{\epsilon}_1, \dots, \hat{\epsilon}_k$. 
Here the case $k=d$ in (\ref{eq:k}) means simply that  $\hat{\epsilon}_{d} \leq q$.

\medskip

If $q = n$,  we get from (\ref{Lag}) that $[\xi]_n =  \tilde{a}^{-n} \ne 0 $ is the dominant term of 
the series $\xi$, hence  the assertion (\ref{eq:k}) holds by Definition \ref{essell-2}.  
Assume that (\ref{eq:k}) holds for some $q>n$. We distinguish two cases: 

\medskip

$\bullet$ \emph{Assume that  $\hat{\epsilon}_{k} <  q$, so $\hat{\epsilon}_k  \leq q-1$.} 
Then, by the induction hypothesis applied to $q-1$, we have that 
the terms of the sequence $\Ess(\xi, m)$ 
which are lower than or equal to 
$q-1$ are precisely $\hat{\epsilon}_0 = {n}, \hat{\epsilon}_1, \dots, \hat{\epsilon}_k$. 
If $[\xi (u) ] _q = 0$, then there is nothing to prove. Assume that $[\xi (u) ] _q \ne 0$.
  Since $  q - n  < \epsilon_{k+1} - m$ by (\ref{eq:k}), the exponent of 
a term appearing in the polynomial $(  \sum_{ m < s < \epsilon_{k+1}} c_s t^{s- m } )^i $
must belong to  the group $\Z\{ n, m, \epsilon_1, \dots, \epsilon_{k}  \}$,
by the definition of the essential exponents $\Ess(\eta, n)$. 
We deduce from this and the right hand side of the equality (\ref{Lag}) 
that  $ q $ must belong to the subgroup  $\Z\{ n, m, \epsilon_1, \dots, \epsilon_{k}  \}$.
Since by definition we have the equality:  
\begin{equation} \label{eq:groups}
\ZÊ\{ n, m, \epsilon_1, \dots, \epsilon_{k}  \} =  \Z\{ n, m, \hat{\epsilon}_1, \dots, \hat{\epsilon}_{k}  \}, 
\end{equation}
$q$ cannot be an essential exponent of $\xi$ with respect to $n$ (see Definition \ref{essell-2}). 

\medskip 

$\bullet$ \emph{Assume that $\hat{\epsilon}_{k} =  q$}.  
Then, by the induction hypothesis applied to $q-1$, we have that 
the terms of the sequence $\Ess(\xi, m)$ 
which are lower or equal to  $q-1$ are precisely 
$\hat{\epsilon}_0 = {n}, \hat{\epsilon}_1, \dots, \hat{\epsilon}_{k-1}$.
We have to prove that the coefficient $[\xi(u)]_q$ does not vanish.  
Notice that there is a term  with exponent equal to $q- n$ 
appearing in the polynomial $(  \sum_{ m < s \leq \epsilon_{k}} c_s t^{s- m } )^i $ 
if and only if $i =1$ and then this term is equal to ${c_{\epsilon_k}}  t^{q- n}$.
Indeed, arguing as in the previous case, we see that 
any other term would provide an expansion of $  \epsilon_{k}  = q- n +m $ in the group 
$\Z\{ n, m, \epsilon_1, \dots, \epsilon_{k-1}  \}$, contradicting the definition of the 
essential exponent $\epsilon_k$.  
Notice that $c_{\epsilon_k} = \tilde{a}^{-m} [\eta]_{\epsilon_k}$, by (\ref{factor}). 
It follows from (\ref{Lag}) that:
\[ 
q \cdot [\xi]_{q} = 
 n \tilde{a}^{-\epsilon_k + m -n}  \left(\frac{- q}{m}\right) \tilde{a}^{-m} [\eta]_{\epsilon_k}  = 
 - q \frac{n}{m}  \tilde{a}^{-\epsilon_k  -n}  [\eta]_{\epsilon_k} ,
 \]
thus $[\xi]_{q} = - \frac{n}{m} \tilde{a}^{-\epsilon_k -n}[\eta]_{\epsilon_k}$ is nonzero.
This finishes the proof of the assertion. 

Theorem \ref{inversioncoef} is proved, since we have also proved the 
inversion formula (\ref{eq: inv3}) for the coefficients.  

\hfill $\Box$

\medskip

 \begin{example}  \label{exinv} 
       Let us consider again the Newton-Puiseux series $\eta(x^{1/4})$ of Example \ref{ex:varnot}. 
           Denote by $\xi(y^{1/6})$ a Newton-Puiseux series 
      corresponding by inversion to $\eta$, through the bijection described 
      in Proposition \ref{propbij}. 
      Applying Proposition \ref{invexpl}, we get:
        $$[\xi(y^{1/6})]_{p/6} = \frac{4}{p} \left[  1 + \sum_{i \geq 1} 
            \binom{-p/6}{i}  (  c x^{\frac{1}{4}})^i \right]_{ -1 + \frac{p}{4}} = 
            \frac{4}{p}   \binom{-p/6}{p-4}    c^{p-4}. $$
      That is: 
          $$\xi(y^{1/6}) = \sum_{p \geq 4} \frac{4}{p}   \binom{-p/6}{p-4}    c^{p-4} y^{p/6}.$$
       The first two exponents in the support $\Supp(\xi(y^{1/6}))$ are therefore $4/6 = 2/3$ 
       and $5/6$, which shows that they constitute the characteristic sequence of 
       $\xi$. The corresponding terms of $\xi$ are, according to the previous formula: 
          $ 1 \cdot y^{2/3}, \:   (-\frac{2}{3} c) \cdot y^{5/6}$.
       One may verify then immediately the correcteness of the formulae stated 
       in the Halphen-Stolz inversion theorem (Corollary \ref{invPuiseux}).         
 \end{example}

\begin{remark}
    In order to compute recursively the coefficients of $\xi(y)$ starting from those of $\eta(x)$, 
     one could also use the method explained by Borodzik \cite{B 12}.  
\end{remark}

\begin{remark}
    We believe that one can use Abhyankar's \cite[First Inversion Theorem, page 111]{A 10} 
    in order to obtain a third proof of the Halphen-Stolz inversion theorem. The approach of 
    that paper seems to be similar in spirit to our first approach. 
 \end{remark}

\section{Generalization to an arbitrary number of variables} \label{arbnumb}

In this section we generalize our first proof of the Halphen-Stolz theorem to an 
arbitrary number of variables. We formulate the needed generalizations of the 
definitions and propositions used in that proof. We only sketch their proofs, 
insisting in the differences with respect to the one-variable case.  
Finally, we explain how our result generalizes  Lipman's inversion theorem for the 
\emph{characteristic exponents} of \emph{quasi-ordinary branches}.

\medskip 

Throughout the section, we consider a fixed number $h \in \N^*$ and we work with the $\Q$-vector 
space $\Q^h$  and various free subgroups of it of rank  $h$, which we will call 
briefly {\bf lattices} of $\Q^h$. We denote by 
$(\nu_1, \nu_2, \dots , \nu_h)$ the canonical basis of $\Q^h$.

\medskip

\subsection{Irreducible exponents of subsets of $\Q^h$ with bounded denominators}
\label{irrexpmany}
$\,$ 
\medskip

\medskip

The notions of \emph{set with bounded denominators} (Definition \ref{defnumsem}) 
and of its \emph{irreducible elements} (Definition \ref{defindec}) extend immediately 
from subsets of $\Q_+$ to subsets of $\Q_+^h$.
If $E$ is such a set, it generates again a semigroup $\N^* \: E \subset \Q_+^h$ 
and a group $\Z \: E \subset \Q^h$. 

Lemma \ref{twoind} (1) remains true in this setting:

\begin{lemma} \label{twoind-S}
If $E \subset \Q_+^h$ is a set with bounded denominators, then 
$\Irr(E) = \Irr(\N^*\: E)$ and this set is the minimal 
    generating set of the semigroup $\N^* \: E$, relative to the inclusion partial order 
    between its generating sets. 
\end{lemma}

Notice that point (2) of  Lemma \ref{twoind}  is not necessarily true 
for $h \geq 2$, as shown by the following standard example: 

\begin{example}  \label{exnonfin} 
   Take $E = (\N^*)^2$. Then $\Irr(E) = \left(\N^* \times \{1\} \right)  \cup 
      \left( \{1\} \times \N^* \right).$ Therefore $\Irr(E)$ is infinite. 
\end{example}

When $h \geq 2$, we will need also to use special order relations 
on the group $(\Q^h, +)$: 

\begin{definition} \label{monord}
   An {\bf additive order} on $\Q^h$ is a partial order relation $\preceq$ on 
   $\Q^h$ satisfying:
     \begin{enumerate}
          \item  $\preceq$ is a total order; 
          \item  if  $ \alpha, \beta, \gamma \in \Z^h$ and $\alpha \preceq \beta$, then 
                      $\alpha + \gamma \preceq \beta + \gamma$.
     \end{enumerate}
The additive order $\preceq$ \textbf {dominates} a set $\theta \subset \Q^h$ if 
           any non empty subset of  $\theta$ with bounded denominators
             has a minimum relative to $\preceq$.              
\end{definition}

\begin{remark} 
If $\preceq$ is an additive order of $\Q^h$,  then there exist an integer $s \in [1, h]$, 
linear forms $u_1, \dots,  u_s \in (\R^h)^*$, and an increasing injective group morphism:
\[
(\Q^h , \preceq) \to (\R^h, \leq_{\mathrm{lex}}), \quad v \to (u_1( v), \dots, u_s (v) ), 
\]
where $\leq_{\mathrm{lex}}$ denotes the lexicographical order (see \cite[Theorem 2.5]{R 86}). 
The lexicographical order is  additive and dominates $\Z_+^h$. 
More generally, if $\preceq$ dominates $\Z_+^h$, then  $\preceq$ defines 
a well-order on $\Z_+^h$, hence 
Definition \ref{monord} is a generalization of the notion of \emph{term order} explained in 
\cite[Chap. 2.2]{CLO 07}. 
\end{remark}

Definition \ref{monord} allows to generalize the notion of \emph{sequence of 
essential elements relative to an integer $p$} (Definition \ref{essell}) in the following way: 

 \begin{definition}  \label{essellgen} 
    Let us consider a set $E \subset \Q_+^h$ with bounded denominators. 
    Let $M$ be a lattice of $(\Q^h, +)$  and $\preceq$ be an additive
    order on $\Q^h$ dominating its subset  $\Q^h_+$.  Then the {\bf sequence} 
    \linebreak 
     $\Ess(E,M, \preceq) := \left(\Ess(E,M, \preceq)_l \right)_{l}$
     {\bf of essential elements of $E$ relative to $M$}  is defined 
     inductively by:
         \begin{itemize}
             \item $\Ess(E,M, \preceq)_0 := \min E$. 
             \item If $l \geq 1$, then the term $\Ess(E,M, \preceq)_l$ is defined if and only if 
                 $ E$ is not included in the group \linebreak 
                     $M + \Z\{\Ess(E,M, \preceq)_0, ...,  \Ess(E,M, \preceq)_{l -1} \}$. In this case: 
               \[  \Ess(E,M, \preceq)_l := \min \left(  E \setminus \left( 
                      M + \Z\{\Ess(E,M, \preceq)_0, \dots, \Ess(E,M, \preceq)_{l -1}  \} \right) \right).  \] 
         \end{itemize}
    \end{definition}

    One gets Definition \ref{essell} by taking $h =1$, $M = p\,  \Z $ and 
    $\preceq$ to be the unique additive order on ($\Q$, +) which dominates $\Q_+$, 
    that is, the usual order. Indeed, then 
      the sequence $\Ess(E, p \, \Z, \preceq)$ defined according to Definition \ref{essellgen} is 
      precisely the sequence $\Ess(E, p)$ defined according to Definition \ref{essell}.

  Lemma \ref{finess} about the finiteness of the sequences $\Ess(E, M)$ 
  holds also in our larger context:

  \begin{lemma}  \label{finessgen}
       Assume that the subset $E \subset \Q_+^h$ has bounded denominators, that $M$ is a 
      lattice of $\Q^h$ and that $\preceq$ is an additive order 
       dominating $\Q^h_+$. 
       Then the sequence of essential exponents $\Ess(E,M, \preceq)$ of $E$ 
       relative to $M$ is finite. 
  \end{lemma}
  \begin{proof}
  For every integer $l \geq 0$ for which $\Ess(E,M, \preceq)_{l}$ is defined, 
   let us denote by $M_l$ the abelian group 
     $M+ \Z\{\Ess(E,M, \preceq)_0, \dots, \Ess(E,M, \preceq)_{l}  \}$.
  Since $(M_l)_l$ is an increasing sequence of abelian groups, the union 
  $ \bigcup_l M_l$ is also an abelian group. 
  The hypothesis that $E$ has bounded denominators implies that this  
  group $\bigcup_l M_l$ is a lattice of $\Q^h$. 
  Any ascending chain of 
  subgroups of a free abelian group of finite rank being stationary,  
  the sequence $(M_l)_l$ must be finite. 
  Therefore, the sequence $\Ess(E,M, \preceq)$ is also finite.
  \end{proof}

    \medskip
    
     If $q \in \mbox{GL}(h, \Q)$  and if $\preceq$ is a additive order  on $\Q^h$, 
        we denote by $\preceq_{q}$ the additive order defined by:  
        \[
        \alpha \preceq_{q} \beta \Leftrightarrow q(\alpha) \preceq q(\beta).
        \]
       By using this notion,  Lemma \ref{ess-P} extends immediately into: 
  
  \begin{lemma}  \label{ess-Pgen}
        Assume that $E \subset \Q_+^h$ has bounded denominators and 
        that $M$ is a lattice of $\Q^h$. Take $q \in \mbox{GL}(h, \Q)$ 
        such that $q(\Q_+^h) \subset \Q_+^h$ and let $\preceq$ be a additive order  
        dominating  $\Q^h_+$.
        Then:
              $$ q \:  (\Ess(E,M, \preceq_{q})) = \Ess \left( q  (E),  q (M) , \preceq \right).$$
  \end{lemma}
  
  Lemmas \ref{essind} and  \ref{eqess1} also extend immediately to our more general context:

   \begin{lemma}  \label{essindgen} 
     The essential elements of a set $E \subset \Q^h_+$ with bounded denominators 
     relative to any lattice $M$  of $\Q^h$ and an additive order $\preceq$ dominating $\Q^h_+$
     are irreducible elements of $E$. 
  \end{lemma}

\begin{lemma} \label{eqessgen}
Let  $E \subset \Q_+^h$ be a set with bounded denominators, $M$ be a lattice of $\Q^h$ 
and $\preceq$ be an additive order dominating $\Q^h_+$. Then we have the following 
equality of essential sequences:
       \[
        \Ess (E,M, \preceq) = \Ess ( \Irr(E), M, \preceq).
        \]
\end{lemma}

  \medskip
  \subsection{On the notions of dual and reciprocal series in several variables}
  \label{dualrecmany}
$\:$ 
\medskip
  
     Consider now the ring $\C[[t_1, t_2, ..., t_h]]$, and its subset:
     $$\C[[t_1, t_2, ..., t_h]]^* \:= \left\{ \phi \in \C[[t_1, t_2, ..., t_h]] \:  : \:  \phi(0, ..., 0) \neq 0  \right\} $$
  consisting of the series with non-zero constant term. It is the group 
  of multiplicatively invertible elements of the ring $\C[[t_1, t_2, ..., t_h]]$. 
  
  If $\phi  \in \C[[t_1, t_2, ..., t_h]]^*$ has constant term $\alpha \neq 0$, then the map: 
    \begin{equation} \label{h-map}
    \begin{array}{rll}
                \C[[t_1, t_2, ..., t_h]]   & \longrightarrow &  \C[[t_1, t_2, ..., t_h]]  \\
               (t_1, t_2, ..., t_h) &  \longrightarrow   &  (t_1 \: \phi(t_1, t_2, ..., t_h), t_2, ..., t_h)
     \end{array} 
     \end{equation}
    is invertible for composition, 
    as its linearization $(t_1, t_2, ..., t_h) \: \to \  (\alpha\:  t_1, t_2, ..., t_h)$ 
    is invertible in $GL(h, \C)$. 
  One has the following generalization of the duality 
  of series in $\C[[t]]^*$, introduced in Definition \ref{dualser}: 

  \begin{definition} \label{dualsergen}
      If $\phi \in \C[[t_1, t_2, ..., t_h]]^*$, then its {\bf dual relative to the first variable}  
      is the unique entire series 
      $\widecheck{\phi} \in \C[[u_1, t_2, ..., t_h,]]^*$ such that the following maps are 
      reciprocal:
      $$(u_1, t_2, ..., t_h) \: \to \:  (u_1\: \widecheck{\phi}(u_1, t_2, ..., t_h), t_2, ..., t_h), 
        \: \: \: \: 
      (t_1, t_2, ..., t_h) \:  \to  \: (t_1 \: \phi(t_1, t_2, ..., t_h), t_2, ..., t_h).$$
       \end{definition}

\begin{remark} Note that the previous definition depends in an essential way on 
   the choice of the first  variable $t_1$, but that it is symmetric in the other variables. 
   If $\phi \in \C[[t_1, t_2, ..., t_h]]^*$,  then setting 
   $u_1= t_1 \: \phi(t_1, t_2, ..., t_h)$ defines a \textit{change of variables} in the ring 
   $\C[[t_1, t_2, ..., t_h]]$. 
   Notice that  $\C[[t_1, t_2, ..., t_h]] = \C[[u_1, t_2, ..., t_h]]$ and by Definition \ref{dualsergen}
   one has the equivalence:
       \begin{equation} \label{f:eqgen}
          u_1= t_1  \: \phi(t_1, t_2, ..., t_h) \Leftrightarrow t_1= u_1 \: \widecheck{\phi} (u_1, t_2, ..., t_h).
       \end{equation}
\end{remark}

The following proposition generalizes Proposition \ref{esspowers} to the case of an 
arbitrary number of variables:

 \begin{proposition} \label{esspowersgen}
          Let $\phi \in \C[[t_1, ..., t_h]]^*$ and $N \in \N^*$. Then: 
            \begin{enumerate}
               \item \label{esspowgen} 
                    $\Irr(\phi^N) = \Irr(\phi)$. Moreover $[\phi^N]_0 = [\phi]_0^N$ and    \,  
                        $[\phi^N]_r = N\:  [\phi]_0^{N-1} [\phi]_r$, for all \,    $r \in \Irr(\phi) \setminus \{0\}$.
               \item \label{essdualgen} 
                      $\Irr(\widecheck{\phi}) = \Irr(\phi)$. Moreover 
                     $[\widecheck{\phi}]_0 = [\phi]_0^{-1}$ \mbox{ and } \,   
                     $[\widecheck{\phi}]_r = -  [\phi]_0^{-{r_1} -2}  [\phi]_r$, for all\,    $r \in \Irr(\phi) \setminus \{0\}$.
            \end{enumerate}
  \end{proposition}

  \begin{proof}
      In what follows, if $k=(k_1, ..., k_h) \in \N^h$, we will write simply:
         $$t^k := t_1^{k_1} \cdots t_h^{k_h}.$$
     One has the following analogue of equation (\ref{explicite}):
         \begin{equation} \label{explicitegen} 
              \phi (t)= \sum_{j \in \Supp(\phi)} [\phi]_j \:  t^j. 
         \end{equation}
      The hypothesis $\phi \in \C[[t_1, ..., t_h]]^*$ translates into $0 \in \Supp(\phi)$, 
      that is, $[\phi]_0 \neq 0$. 
       
       \medskip

       \noindent
           (\ref{esspowgen}) \emph{Consider first the case of $\phi^N$}.  
           By equation (\ref{explicitegen}), 
          we get the exact analogue of the expansion (\ref{explicitepow}):
             \begin{equation} \label{explicitepowgen} 
                   \phi^N (t)= \sum_{j_1, ..., j_N \in \Supp(\phi)} [\phi]_{j_1} \cdots  [\phi]_{j_N} 
                       t^{j_1 + \cdots + j_N}.   
             \end{equation}   
          Then the proof is identical to that of the one-variable case.

                 \medskip
                 \noindent
                 (\ref{essdualgen}) 
           \emph{Consider now the case of $\widecheck{\phi}$}. 
           Write the analogue 
          of the expansion (\ref{explicitegen}) for the series $\widecheck{\phi}$:
              \begin{equation} \label{explicitedualgen} 
                    \widecheck{\phi}(u) = \sum_{k \in \Supp(\widecheck{\phi})} [\widecheck{\phi}]_k \:  u^k,
              \end{equation}
  \noindent where:  
      $$u^k :=u_1^{k_1}t_2^{k_2}\cdots  t_h^{k_h} \mbox{ for all } k=(k_1,\ldots,k_h)\in \N^h.$$        
         By Definition   \ref{dualsergen},  if  $\widecheck{\phi}(u_1, t_2, ..., t_h)$ is 
         the dual with respect to $t_1$ of the series $\phi(t_1, t_2, ..., t_h)$, 
         then one has the identity:       
             \[t_1 =  (t_1 \: \phi(t_1, t_2, ..., t_h) ) \: \cdot \: 
               \widecheck{\phi}( t_1\:  \phi(t_1, ..., t_h), t_2, ..., t_h ) \]
        which, after division by $t_1$ and combination with the expansion 
        (\ref{explicitedualgen}), gives: 
          \[
                 1 =  \sum_{k \in \Supp(\widecheck{\phi})} [\widecheck{\phi}]_k \:  t^k \phi(t)^{k_1+1}.   
          \]
         Expand now the powers $ \phi(t)^{k_1 +1}$ using equation (\ref{explicitepowgen}). We get:
           \[ 1 = \sum_{\small{\begin{array}{c}
                                      k \in \Supp(\widecheck{\phi})\\
                                      j_{1}, ..., j_{k_1 +1} \in \Supp(\phi)
                                  \end{array}} }    
                          [\widecheck{\phi}]_k  \:  [\phi]_{j_1} \cdots [\phi]_{j_{k_1+1}} \: t^{k + j_1 + 
                              \cdots + j_{k_1+1}}.   \]
          Therefore:
           \begin{equation}  \label{vanish-S} 
                \sum_{\small{\begin{array}{c}
                                      k \in \Supp(\widecheck{\phi})\\
                                      j_{1}, ..., j_{k_1 +1} \in \Supp(\phi) \\
                                      k + j_{1} + \cdots + j_{k_1+1} = p
                                  \end{array} }}    
                          [\widecheck{\phi}]_k  \:  [\phi]_{j_{1}} \cdots [\phi]_{j_{k_1 +1}}    =0 , \:  
                             \hbox{\rm for all } \: p \in \N^h \setminus \{ 0 \}.
           \end{equation}

           $\bullet$ \emph{Let us show first that $\Irr(\phi) =  \Irr(\widecheck{\phi})$}. 
         Using Lemma \ref{twoind-S},  we reason 
         as in the one variable case. We must take into account that for $h \geq 2$, the element 
           $r \in  \N^* ( \Supp(\widecheck{\phi})) \: \setminus \:  \N^* (\Supp(\phi))$ 
           chosen to be minimal with this property (for the componentwise partial order)
           is not necessarily unique. Let us choose
           $r$ to be the smallest element with this property, relative to 
           an additive order $\preceq$ dominating $\Q^h_+$.
           Then the proof of the assertion follows  exactly  by the same argument as for $h=1$.

            \medskip
           $\bullet$ \emph{We prove the identities relating the coefficients 
              associated to the irreducible exponents of $\phi$ and $\widecheck{\phi}$. }
              If $r \in \Irr(\phi) =  \Irr(\widecheck{\phi})$ and if: 
              $k + j_1 + \cdots + j_{k_1 +1} = r$  for 
              $k \in \Supp(\widecheck{\phi})$ and $j_1, \dots  j_{k_1 +1} \in \Supp(\phi)$, 
              then we obtain, by the same argument as in the one variable case, that 
              $k = r$ or $k =0$.
            We deduce the 
              following analogues of equation (\ref{eqcoefsym}): 
               \begin{equation} \label{eqcoefsymgen} 
                  [\widecheck{\phi}]_0 \:  [\phi]_r +  [\widecheck{\phi}]_r \:   [\phi]_0^{r_1+1} =0 
               \end{equation}
                from which one gets the stated equality between coefficients of terms 
                with irreducible exponents.   
  \end{proof}

 Combining Proposition \ref{esspowersgen} with Lemma  \ref{eqessgen}, we obtain 
 the  following extension of Corollary \ref{coress}:

  \begin{corollary} \label{coressgen}
      Let $\phi \in \C[[t_1, t_2, \dots, t_h]]^*$, $N\in \N^*$ and $M$ be any lattice of 
      $\Q^h$. 
       Then the sequences of essential exponents 
       of $\phi, \phi^N$ and $\widecheck{\phi}$ 
       relative to $M$ coincide. 
  \end{corollary}

\medskip
\subsection{Newton-Puiseux series in several variables}
\label{NPseriesmany}
$\:$ 
\medskip

  We will consider the following analogue of the ring of Newton-Puiseux series in one variable: 
     $$  \C [[x_1^{1/ \N}, x_2^{1/ \N},..., x_h^{1/ \N} ]] := \bigcup_{n_i \in \N^*,\; 1\leq i\leq h} 
                 \C [[x_1^{1/{n_1}}, x_2^{1/ {n_2}},..., x_h^{1/ {n_h}} ]].$$
  We  say that its elements are \emph{Newton-Puiseux series in the variables 
  $x_1, ..., x_h$}.  
  The \emph{support} $S(\eta)$ of such a series $\eta$ is a subset with 
  bounded denominators of $\Q_+^h$.

\begin{definition}  \label{NPsolgen}
    Assume that $f(x_1 , y_1, x_2, \dots,  x_h) \in \C[[x_1, y_1, x_2,\dots, x_h]]$ has vanishing 
    constant term. A {\bf Newton-Puiseux series of $f$ relative to $(x_1, x_2, \dots , x_h)$} 
    is a series: 
      \[ 
      \psi \in \C[[ x_1^{1/ \N}, x_2^{1/ \N}, \dots , x_h^{1/ \N} ]]
      \]
      such that 
    $f(x_1,  \psi, x_2, \dots,  x_h) =0$. The series $\psi$ is called {\bf $x_1$-dominating} if 
    it is of the form: 
      \[a \cdot x_1^{\lambda} (1 +  \mbox{higher order terms}  ),
      \]
   where $\lambda \in \Q_+^*$, and $a \in \C^*$. 
 A representation of $\psi \in \C[[ x_1^{1/ \N}, x_2^{1/ \N},\dots, x_h^{1/ \N} ]] $ of the form:
      \begin{equation} \label{h-eta}
         \psi = \eta (x_1^{1/n_1}, x_2^{1/n_2} , \dots, x_h^{1/n_h}) \mbox{ with }   
          \eta \in \C[[ t_1, t_2, \dots , t_h]]   
      \end{equation}
   is  called \textbf{primitive} if 
 it is primitive in each variable separately in the sense of Definition \ref{def: prim}. 
 \end{definition}

\begin{example} \label{exdom}
  The Newton-Puiseux series 
             $x_1^{3/2}  + x_1^{7/4} \:  x_2^{1/2} - 2 x_1^2 \: x_3^{1/3} $
   is $x_1$-dominating. But the series:
       $ x_1^{1/2}  + x_2^{1/2}$ and $x_2^{1/3}  + x_1 \: x_2^{2/3} $
   are not $x_1$-dominating. Instead, the second one is $x_2$-dominating. 
\end{example}

If  $ \psi \in \C[[ x_1^{1/ \N}, \dots , x_h^{1/ \N} ]]$, then there is a series  
 $f \in  \C[[x_1, y_1, x_2, \dots, x_h]]$ such that 
  $f(x_1,\psi,  x_2, \dots, x_h) =0$. 
 Indeed, one can get such an $f$ in the ring $\C[[x_1,  \dots, x_h]][y_1]$  
 (cf. Remark \ref{Galact}  in the $1$-variable case):

\begin{remark}
 Recall that for any $n \in \N^*$, we denote by $G_n$ the subgroup of $(\C^*, \cdot)$ 
consisting of the $n$-roots of unity.
Consider $n_1, \dots, n_h \in \N^*$. Let $\C ((x_1^{1/n_1}, \dots, x_h^{1/n_h}))$ 
be the fraction field of 
$\C[[x_1^{1/n_1}, \dots, x_h^{1/n_h} ]]$.  The field extension 
$\C((x_1, \dots, x_h)) \subset \C ((x_1^{1/n_1}, \dots, x_h^{1/n_h}))$ is finite and Galois. 
Its Galois group is isomorphic to $G_{n_1} \times \cdots \times G_{n_h}$, acting 
on $\C ((x_1^{1/n_1}, \dots, x_h^{1/n_h}))$
by:  
\[
     \left( (\rho_1, \dots, \rho_h) , x_1^{a_1/n_1} \cdots x_h^{a_h/n_h} \right) \to \rho_1^{a_1} 
      \cdots \rho_h^{a_h} \cdot  x_1^{a_1/n_1} \cdots x_h^{a_h/n_h} .
\]
If $\psi \in \C[[x_1^{1/n_1}, \dots, x_h^{1/n_h} ]]$ is a Newton-Puiseux series,  then 
the field extension 
$
\C ((x_1, \dots, x_h)) \subset \C((x_1, \dots, x_h))[\psi] 
$
is finite and  its Galois group $G$ is isomorphic to the quotient of 
$G_{n_1} \times \cdots \times G_{n_h}$  
by its subgroup formed by those elements which leave $\psi$ fixed.  If $n=|G|$ and 
$ \psi_1= \psi, \psi_2, \dots, \psi_n $ are the different conjugates of $\psi$  
under the action of the group $G$,  then the polynomial: 
\[
f = \prod_{j=1}^n (y- \psi_j) \in  \C[[x_1^{1/n_1}, \dots, x_h^{1/n_h} ]][y]
\]
is invariant under the action of $G_{n_1} \times \cdots \times G_{n_h}$ on its coefficients.
It follows that $f$ must belong to $\C[[x_1, \dots, x_h ]][y]$ and that $\psi$ is a Newton-Puiseux series relative to $f$. 
\end{remark}

\begin{remark} 
Let $f \in \C[[x_1,  \dots, x_h]][y_1]$ be an irreducible polynomial such that its discriminant 
$\Delta_{y_1} f $ is the product of  a monomial and of
a unit in the ring  $\C[[x_1, \dots, x_h]]$. Then, by the Jung-Abhyankar theorem, all the roots of $f$ 
are Newton-Puiseux series in the variables $x_1, \dots, x_h$ (see \cite{A 55}).
Let us mention that the roots obtained in this way have special properties, for instance, 
the Newton-Puiseux series 
$x_1^{3/2} + x_2^{5/2}$ cannot be a root of  the polynomial $f$ (see Lemma \ref{l:Lipman}).
Notice also that if the discriminant of $f$ is not of this form, the roots may not be 
expressible as Newton-Puiseux series in the variables $x_1, \dots, x_h$. 
An example of this last phenomenon is the polynomial $f = x_1^3+ x_2^3 + y_1^2$.
 \end{remark}

   \medskip
   \subsection{The generalized Halphen-Stolz inversion theorem}
   \label{genhalphmany}
   $\:$ 
   \medskip

  In this subsection we assume that  $\psi$  is a  $x_1$-dominating Newton-Puiseux power series
with primitive representation (see Definition \ref{NPsolgen}):
\[
 \psi = \eta (x_1^{1/n_1}, x_2^{1/n_2} , \cdots, x_h^{1/n_h}).
\]
Then the dominating 
term $a \cdot x_1^{\lambda}$ of $\psi$ satisfies:
  $$\lambda = \dfrac{m_1}{n_1},$$
with $m_1 \in \N^*$. The series  $\eta$  is therefore of the form:
     $$\eta(t_1, t_2, ..., t_h) =  a \cdot t_1^{m_1} (1 + \mbox{higher order terms}),$$
  with $m_1 >0$ and $a \in \C^*$. 
 Let us choose an $m_1$-th root $\tilde{a}\in \C^*$ of $a$.
 Then, we have a unique $m_1$-th root 
 $t_1 \: \tilde{\eta}(t_1, t_2, ..., t_h) \in \C[[t_1, t_2, ..., t_h]]$ of $\eta$:
    \begin{equation}  \label{powergen}
        \eta(t_1, t_2, ..., t_h) = (t_1 \:  \tilde{\eta}(t_1, t_2, ..., t_h))^{m_1}, 
    \end{equation}
 with constant term $[\tilde{\eta}]_0 = \tilde{a} \:  \ne 0$.

 \begin{example}  \label{ex: varnot-S} 
     Start from the Newton-Puiseux series:
       $\psi=  x_1^{3/2}  + x_1^{7/4} \:  x_2^{1/2} - 2 \: x_1^2 \: x_3^{1/3}$. 
       We get $a =1$, $n_1 =4$, $n_2 =2$, $n_3= 3$ and 
       $\psi = \eta( x_1^{1/4}, x_2^{1/2} , x_3^{1/3})$ where
        ${\eta}(t) = t_1^6 + t_1^7 t_2 - 2 \:  t_1^8 t_3= t_1^6(1 + t_1   t_2 - 2 \:  t_1^2 t_3).$
     This shows that $m_1 = 6$ and if $\tilde{a} = 1$ then:
         $\tilde{\eta}(t) = (1 +  t_1   t_2 - 2 \:  t_1^2 t_3   )^{1/6} := 
              1 + \sum_{k \in \N^*} \binom{1/6}{k}  ( t_1   t_2 - 2 \:  t_1^2 t_3  )^k$.
 \end{example}

 Let us come back to the general case. 
 Denote by $\tilde{\xi}(u_1, t_2, ..., t_h) \in \C[[u_1, t_2, ..., t_h]]^*$ the dual series of 
 $\tilde{\eta}(t_1, t_2, ..., t_h)$ with respect to $t_1$ (see Definition \ref{dualsergen}).  
 Hence, one has the equivalence (see (\ref{f:eqgen})):
    \begin{equation} \label{invsergen}
           u_1 = t_1 \:  \tilde{\eta}(t_1, t_2, ..., t_h) \:  \Leftrightarrow \:   
               t_1 = u_1 \: \tilde{\xi}(u_1, t_2, ..., t_h). 
    \end{equation}

 \medskip

 We know that there exists a series $ f (x_1,y_1,  x_2, \dots, x_h, ) \in \C[[ x_1, y_1, x_2, \dots, x_h]]$ 
 such that:
   \[f(x_1, \eta(x_1^{1/n_1},  x_2^{1/n_2}, ..., x_h^{1/n_h}), x_2, ..., x_h) =0. \]
 Replacing each $x_i$ by $t_i^{n_i}$ and using the equality (\ref{powergen}), we get: 
        \begin{equation}   \label{eqsm1gen}
              f(t_1^{n_1}, (t_1 \tilde{\eta}(t_1, t_2, ..., t_h))^{m_1}, t_2^{n_2}, ..., t_h^{n_h}) =0. 
         \end{equation} 
 By doing the second change of variable of formula (\ref{invsergen}), we deduce that: 
     \begin{equation}  \label{eqsm2gen} 
            f((u_1 \: \tilde{\xi}(u_1, t_2, ..., t_h))^{n_1}, u_1^{m_1}, t_2^{n_2}, ..., t_h^{n_h}) =0. 
     \end{equation}
 Consequently, if one defines: 
    \begin{equation}  \label{powerbisgen}
        \xi(u_1, t_2, ..., t_h)  = (u_1 \:  \tilde{\xi}(u_1, t_2, ..., t_h))^{n_1}  
    \end{equation} 
 (an equation which is analogous to (\ref{powergen})), then one sees that:
      \[   f(\xi(y_1^{1/m_1}, x_2^{1/n_2}, ...,x_h^{1/n_h} ), y_1, x_2, ..., x_h) =0, \]
 that is,  $ {\xi}(y_1^{1/m_1}, x_2^{1/n_2}, ...,x_h^{1/n_h} ) $
  is a Newton-Puiseux series of $f(x_1, y_1, x_2, ..., x_h)$ with respect 
   to the variables $(y_1, x_2, ..., x_h)$ (see Definition \ref{NPsolgen}).

\medskip

 Recall that 
    we denote by 
$(\nu_1, \nu_2, \dots , \nu_h)$ the canonical basis of $\Q^h$. 
The following lemma is a multi-variable analogue of Lemma \ref{divpower}, 
formulated using Definition \ref{monord}:

    \begin{lemma} \label{divpowergen} 
    Let $\preceq$ be an additive order of $\Q^h$ dominating $\Q^h_+$.
        Denote:
           \begin{equation} \label{eqess-S}
                \left \{   \begin{array}{l}
                                    \Ess(\eta, \:   n_1\Z \nu_1 + \Z \nu_2 + \cdots + \Z \nu_h, \:  \preceq ) = 
                                           (m_1 \nu_1 , \epsilon_1,  ..., \epsilon_d), \\
                                    \Ess(\xi, \:  m_1 \Z \nu_1 + \Z \nu_2 + \cdots + \Z \nu_h, \:  \preceq) = 
                                            (n_1 \nu_1, \epsilon'_d, ..., \epsilon'_{d'}). 
                             \end{array}   \right. 
           \end{equation}
        Then: 
           \[   \left \{   \begin{array}{l}
                                    \Ess(\tilde{\eta}, \:  \gcd (n_1,m_1) \Z \nu_1 + \Z \nu_2 + 
                                          \cdots + \Z \nu_h, \:   \preceq) = 
                                           (0, \epsilon_1 - m_1 \nu_1, ..., \epsilon_d - m_1 \nu_1), \\
                                    \Ess(\tilde{\xi}, \:   \gcd (n_1,m_1) \Z \nu_1 + \Z \nu_2 + \cdots + \Z \nu_h, 
                                           \:  \preceq) = 
                                        (0, \epsilon'_1 - n_1 \nu_1, ..., \epsilon'_{d'} - n_1 \nu_1). 
                             \end{array} \right.  \]
    \end{lemma}
    
    \begin{proof}
         By symmetry, we may treat only the case of the series $\tilde{\eta}$. 
            As $\tilde{\eta} \in \C[[t_1, t_2, ..., t_h]]^*$, 
           Proposition \ref{esspowersgen} implies that 
     $\Irr( \tilde{\eta}) = \Irr ( \tilde{\eta}^m )$. 
      Combining this with Lemma \ref{eqessgen}, we see that for any lattice $M$ of $\Q^h$ one has: 
   \[
   \Ess ( \tilde{\eta}, M, \preceq) = \Ess ( \Irr( \tilde{\eta}), M, \preceq) =  
    \Ess (  \Irr ( \tilde{\eta}^m), M, \preceq ) =  
   \Ess ( \tilde{\eta}^m, M, \preceq ).
   \]
Thus, it is enough to prove that $(0, \epsilon_1 - m_1 \nu_1, ..., \epsilon_d - m_1 \nu_1)$ is 
       the sequence of essential exponents 
        of $\tilde{\eta}^{m_1}$ relative to $\gcd(n_1,m_1)\Z \nu_1 + \Z \nu_2 + \cdots + \Z \nu_h$ and 
        the chosen additive order. 
          By formula (\ref{powergen}), we get $\tilde{\eta}^{m_1} =  t_1^{-m_1} {\eta}$. Therefore        
         $ \Supp(\tilde{\eta}^{m_1}) = \Supp({\eta}) - m_1 \nu_1$.  
       Using Definition \ref{essellgen}, we see that we are done if we prove that:
           \begin{itemize}
             \item $\min (\Supp({\eta}) - m_1 \nu_1) = 0$. 
             \item For all $k \in \{1, ..., d\}$: 
               \[  \epsilon_k - m_1 \nu_1 = \min \left((\Supp({\eta}) - m_1 \nu_1) \setminus  \left(
                      \gcd (n_1,m_1)\Z \nu_1 + \Z \nu_2 + \cdots + \Z \nu_h + 
                            \Z\{ \, 0,  \epsilon_1 - m_1 \nu_1, ..., \epsilon_{k-1} - m_1 \nu_1 \} \right) \right).   \] 
             \item $\Supp({\eta}) - m_1 \nu_1   \subset 
                   \gcd (n_1,m_1) \Z \nu_1 + \Z \nu_2 + \cdots + \Z \nu_h + 
                         \Z\{\,  0, \epsilon_1 - m_1 \nu_1, ..., \epsilon_d - m_1 \nu_1 \}$,
        \end{itemize}
        where the minimum is taken with respect to the additive order $\preceq$.
        But all these facts are immediate from the definition of the essential exponents 
            $\epsilon_i$, because:
        \[ \begin{array}{c}
             \gcd (n_1,m_1) \Z \nu_1 + \Z \nu_2 + \cdots + \Z \nu_h +   
                    \Z\{ 0,  \epsilon_1 - m_1 \nu_1, ..., \epsilon_{k-1} - m_1 \nu_1 \}  =  \\ 
                  = n_1\Z \nu_1 + \Z \nu_2 + \cdots + \Z \nu_h +   
                    \Z \left\{ m_1 \nu_1,  \epsilon_1 , ..., \epsilon_{k-1}  \right\},
             \end{array}
        \]
         for all $1 \leq k \leq d$, an equality which may be proved immediately by double inclusion.
    \end{proof}

\smallskip

Our extension of Theorem \ref{inversioncoef} 
to the case of an arbitrary number of variables follows then exactly as in the one variable case:

  \begin{theorem}   \label{inversioncoefgen}
       Let $\tilde{\eta}\in \C[[t_1, t_2, ..., t_h]]^* $ and $\tilde{\xi} \in \C[[u_1, t_2, ..., t_h]]^*$ 
       be dual relative to the first coordinate
       and consider $m_1,n_1 \in \N^*$. Let $\tilde{a}$ be the constant term of $\tilde{\eta}$. Denote 
       by $(\nu_1, \nu_2, ..., \nu_h)$ the canonical basis of the free abelian group $\Z^h$. 
       Introduce the $t_1$-dominating and $u_1$-dominating series:
            \[ \left\{ \begin{array}{l} 
                    \eta(t_1, t_2, ..., t_h) = (t_1 \: \tilde{\eta}(t_1, t_2, ..., t_h))^{m_1}, \\
                     \xi(u_1, t_2, ..., t_h) = (u_1\:  \tilde{\xi}(u_1, t_2, ..., t_h))^{n_1}, 
                \end{array} \right. \]
        and their sequences of essential exponents relative to an additive order $\preceq$ dominating 
        $\Q^h_{\geq 0}$:  
         \[  \left \{   \begin{array}{l}
                                    \Ess(\eta, n_1 \Z \nu_1 + \Z \nu_2 + \cdots + \Z \nu_h, \preceq) = 
                                         (m_1 \nu_1 , \epsilon_1, ..., \epsilon_d), \\
                                    \Ess(\xi, m_1 \Z \nu_1 + \Z \nu_2 + \cdots + \Z \nu_h, \preceq) = 
                                          (n_1 \nu_1, \epsilon'_1, ..., \epsilon'_{d'}). 
                             \end{array}   \right.  \]
   Then one has the following inversion formulae for exponents and coefficients, 
   where we denote by $\epsilon_{k,1}$ the first coordinate of $\epsilon_k \in \Q_+^h$ (that is, 
    the coefficient of $\nu_1$ in the expansion 
           $\epsilon_k = \sum_{i=1}^h \epsilon_{k,i} \nu_i$):
   \begin{equation}  \label{eq: inv1-S}
          d' = d, 
    \end{equation}
  \begin{equation}  \label{eq: inv2-S}
   \epsilon_k' + m_1 \nu_1 = \epsilon_k + n_1 \nu_1, \: \hbox{\rm for all } \:  k \in \{1, ..., d \}, 
 \end{equation}
         \begin{equation} \label{eq: inv3-S}
                  [{\xi}]_{n} =  {\tilde{a}}^{-n_1} 
            \quad 
            \mbox{ and }  \quad 
              [{\xi}]_{\epsilon_k'} = - \frac{n_1}{m_1} \:  {\tilde{a}}^{-n_1 - \epsilon_{k, 1}} 
              [{\eta}]_{\epsilon_k}, \: \hbox{\rm for all } \:  k \in \{1, ..., d \}. 
          \end{equation}
      \end{theorem}

As a consequence of this theorem, we get the following generalization 
of the Halphen-Stolz inversion theorem (Corollary \ref{invPuiseux}):

\begin{corollary}[{\bf The generalized Halphen-Stolz inversion theorem}]  \label{invPuiseuxgen}
        $\,$
        
        \noindent
          Let  $\eta (x_1^{1/n_1}, x_2^{1/n_2}, ..., x_h^{1/n_h}) $ and 
          $\xi(y_1^{1/m_1}, x_2^{1/n_2}, ..., x_h^{1/n_h}) $ be Newton-Puiseux series of 
          $f(x_1, y_1, x_2, ..., x_h)$ 
            relative to $(x_1, x_2, ..., x_h)$ and $(y_1, x_2, ..., x_h)$ 
          respectively. 
          As before, we assume that 
              $\eta(t_1, t_2, ..., t_h)  = (t_1 \:  \tilde{\eta}(t_1, t_2, ..., t_h))^{m_1}$ and   
              $\xi(u_1, t_2, ..., t_h)  = (u_1 \:  \tilde{\xi}(u_1, t_2, ..., t_h))^{n_1}$,   
              where $\tilde{\eta}(t_1, t_2, ..., t_h)$ and $\tilde{\xi}(u_1, t_2, ..., t_h)$ are dual 
              relative to the first coordinate   and $[\tilde{\eta}]_0 = \tilde{a}$. 
          Let $\preceq$ be a fixed additive order dominating $\Q^h_+$. Denote:
                \[  \left \{   \begin{array}{l}
                                    \Ess(\eta (x_1^{1/n_1}, x_2^{1/n_2}, ..., x_h^{1/n_h}), \Z^h, \preceq) = 
                                          \left(\dfrac{m_1}{n_1} \nu_1, e_1, \dots, e_d \right), \\
                                    \Ess(\xi( y_1^{1/m_1}, x_2^{1/n_2}, ..., x_h^{1/n_h}), \Z^h, \preceq) = 
                                              \left(\dfrac{n_1}{m_1} \nu_1, e_1', \dots, e'_{d'} \right). 
                             \end{array}   \right.  \]            
          Then one has the following inversion formulae for exponents and coefficients, 
           where we denote by $e_{k,1}$ the first coordinate of $e_k \in \Q_+^h$ (that is, 
    the coefficient of $\nu_1$ in the expansion $e_k = \sum_{i=1}^h e_{k,i} \nu_i$): 
               \begin{equation} \label{eq: invP0-S}
                 \displaystyle{ d'=d  }. 
           \end{equation} 
           \begin{equation} \label{eq: invP1-S}
                 \displaystyle{ m_1 (\nu_1 + e_k') = n_1 (\nu_1 + e_k)  } \,  \mbox{  for  all } \,  k \in \{1, ..., d\}. 
           \end{equation} 
           \begin{equation} \label{eq: invP2-S}
                 [\xi (y_1^{1/m_1},\ldots) ]_{n_1/m_1} = \tilde{a}^{-n_1} 	
                 \,                   
                 \mbox{ and }
                \,   \,  
                 [\xi (y_1^{1/m_1},\ldots) ]_{e_k'} = - \dfrac{n_1}{m_1} \:  
                      \tilde{a}^{-( 1  + e_{k, 1})n_1} 
                  [\eta (x_1^{1/n_1},\ldots)]_{e_k} \,  \mbox{  for  all } \,  k \in \{1, \ldots, d\}. 
             \end{equation} 
          
    \end{corollary}

\medskip
     
      In the case in which $\tilde{a} = 1$, the inversion formula for 
    the coefficients stated in  Corollary \ref{invPuiseuxgen} may be written 
    in a more symmetric way, easier to remember, and analogous to Corollary 
    \ref{corone}:
      
    \begin{corollary}  \label{coronegen}
        Assume moreover that the constant coefficient $\tilde{a}$ of $\tilde{\eta}$ is equal to $1$.
        Then: 
                 \[  
                          [\xi  ]_{n_1/m_1} = 1 = 
                              [\eta ]_{m_1/n_1},  \quad \mbox{ and }  \quad  
                 m_1 [\xi  ]_{\epsilon_k'} + 
                     n_1  [\eta ]_{\epsilon_k} =0  
                        \mbox{ for all } k \in \{1, ..., d\}. 
                      \]
    \end{corollary}

In order to summarize our reasoning, let us draw the analogue of the flow-chart (\ref{reasoning}) 
in which $f(x_1,y_1, x_2, \dots , x_h) \in \C[[x_1, y_1, x_2, ..., x_h]]$ is an irreducible series:

 \begin{equation} \label{reasoninggen}
    \xymatrix{
                    & \overset{  \Ess(\eta(x_1^{1/n_1}, \dots), \Z^h)  }{ \{   \eta(x_1^{1/n_1}, \dots) \}  }
                    \ar@{<->}[rr]_(.48){(\mathrm{Lm.} \ref{ess-Pgen})}^(.48){x_1={t_1}^{n_1}}      &  & 
                      \overset{  \Ess(\eta(t_1, \dots), n_1 \Z \nu_1 + \cdots )}{  \{ \eta(t_1, t_2, \dots) \}  }     &  & 
                     \overset{ \Ess(\tilde{\eta}(t_1, \dots), \gcd(m_1,n_1) \Z  \nu_1 + \cdots)}{  
                         \{  \tilde{\eta}(t_1, \dots)  \} }\ar[ll]^(.55){\begin{psmallmatrix}
                                     (\mathrm{Prop.} \ref{esspowersgen} (\ref{esspowgen}),  \\
                                     \mathrm{Lm.} \ref{divpowergen}) 
                                     \end{psmallmatrix}}_(.55){\eta=(t_1 \tilde{\eta})^{m_1}} 
                             \ar@{<->}[dd]_{(  \mathrm{Prop.} 
                                  \ref{esspowersgen} ( \ref{essdualgen})  )}^{\begin{array}{c}
                                                  u_1 = t_1 \:  \tilde{\eta}  \\  
                                                  \Updownarrow \\
                                                  t_1 = u_1 \:  \tilde{\xi}
                                             \end{array} }\\
        \hspace{-0.4cm} f(x_1,y_1, x_2, \dots , x_h) \; 
         \hspace{-0.9cm}\ar[ur] \hspace{-0.9cm}\ar[dr]     
                &     &     &   \\
                    & \underset{\Ess(\xi(y_1^{1/m_1}, \dots), \Z^h)} { \{  \xi(y_1^{1/m_1}, \dots)   \} }
                          \ar@{<->}[rr]^(.48){(\mathrm{Lm.} \ref{ess-Pgen})}_(.48){y_1 ={u_1}^{m_1}}      &      &  
                      \underset{ \Ess(\xi(u_1, \dots), m_1 \Z  \nu_1+ \cdots )}{  \{ \xi(u_1, t_2, \dots) \} } &   & 
                 \underset{ \Ess(\tilde{\xi}(u_1, \dots), \gcd(m_1, n_1) \Z  \nu_1+ \cdots)}{ 
                     \{ \tilde{\xi}(u_1, \dots) \} } \ar[ll]_(.55){\begin{psmallmatrix}
                                     (\mathrm{Prop.} \ref{esspowersgen} (\ref{esspowgen}),  \\
                                    \mathrm{Lm.} \ref{divpowergen})
                                     \end{psmallmatrix}}^(.55){\xi=(u_1 \tilde{\xi})^{n_1}} 
               }
    \end{equation}

\medskip
\subsection{The special case of quasi-ordinary series}
\label{comparmany}
$\:$ 
\medskip

Among the Newton-Puiseux series in several variables, the \emph{quasi-ordinary} ones 
form a distinguished subclass, having many special properties. We compare both classes in this section 
with the help of additive orders and toric modifications. In particular, we get that Lipman's inversion theorem 
for quasi-ordinary series can be seen as a particular case of our generalized inversion theorem (Corollary   \ref{invPuiseuxgen}).

\begin{definition}
A series $\psi \in \C[[ x_1^{1/ \N}, \dots, x_h^{1/ \N} ]] $ is 
\textbf{quasi-ordinary} if $\psi$ is a Newton-Puiseux series relative to an irreducible polynomial  
$f \in \C[[x_1, \dots, x_h]][y_1]$, such that 
the discriminant, $\Delta_{y_1} (f) \in \C[[x_1, \dots, x_h]]$ of $f$ with respect to $y_1$, 
is the product of a monomial and of a unit in the ring $\C[[x_1, \dots, x_h]]$.
\end{definition}

If  the discriminant 
 $\Delta_{y_1} (f) $ of $f \in \C[[x_1, \dots, x_d]][y_1]$ 
is a monomial times a unit then, by the Jung-Abhyankar theorem (see \cite{A 55} and \cite{J 08}),  
$f$ factors in the ring
$\C[[ x_1^{1/ \N}, \dots , x_h^{1/ \N} ]][y_1]$ as a product of polynomials of degree 
$1$ in the variable $y_1$. 
If $y_1 - \psi$, $y_1 - \psi'$ are two different factors of $f$ in this ring, 
then $\psi - \psi'$ divides the discriminant $\Delta_{y_1} (f)$, hence
$\psi - \psi'$ is the product of a monomial times a unit in the ring 
$\C[[ x_1^{1/ \N}, \dots , x_h^{1/ \N} ]]$.
The monomials obtained in this way:
\[
x_1^{\alpha_{k,1}} \, \cdots \, x_{h}^{\alpha_{k,h}}, \mbox{ for } k \in \{1, \dots, g \},
\]
 are called  the {\bf characteristic monomials}, and the tuples:  
\[
\alpha_k  = (\alpha_{k,1}, \dots, \alpha_{k,h}), \mbox{ for } k \in \{1, \dots, g \},
\]
the {\bf characteristic exponents} of the quasi-ordinary series $\psi$. 
Lipman showed that the characteristic exponents
determine many features of the geometry of the germ 
of hypersurface defined by $f$ 
(for precise definitions and related results, see for instance \cite{L 65, L 83, L 88A, G 88}).
He also proved the following combinatorial characterization of quasi-ordinary power series 
(see \cite[Proposition 1.5]{L 65}  and  \cite[Proposition 1.3]{G 88}):

\medskip

\begin{lemma} \label{l:Lipman}
 Denote by $\leq$  the coordinate-wise order on  $\Q^h$. 
A  series $\psi \in \C[[ x_1^{1/ \N}, \dots, x_h^{1/\N} ]] $ is
quasi-ordinary if and only if there exist an integer $n \geq 1$ and elements 
$\lambda_1, \dots, \lambda_r \in \Supp(\psi)$
such that:
\begin{enumerate}
\item The support $\Supp(\psi)$ is included in $\frac{1}{n} \Z^h_+$. 
\item Every  $\lambda \in  \Supp(\psi)$ belongs to the group 
     $\Z^h + \sum_{\lambda_j \leq \lambda} \Z \lambda_j$.
\item  $\lambda_{i} \leq \lambda_{i+1}$, for every $i \in \{1, \dots, r-1\}$.
\item $\lambda_{i}$ does not belong to the group $ \Z^h + \sum_{j \leq i - 1} \Z \lambda_j $, 
for every $i \in \{1, \dots, r\}$.
\end{enumerate}
If such elements  exist, then they are the characteristic exponents of $\psi$.
\end{lemma}

\medskip

The following lemma is an analogue of Lemma \ref{charess}. Its proof 
is a consequence 
of the definitions of essential 
exponents 
and of Lemma \ref{l:Lipman}. It shows how to recover the
characteristic exponents of a quasi-ordinary series from a sequence of essential 
exponents relative to
 the lattice $\Z^h$ and \emph{any} additive order $\preceq$ dominating $\Q^h_+$.

\begin{lemma}  \label{charess-S} 
    Let $\psi \in \C[[ x_1^{1/ \N},\dots, x_h^{1/ \N} ]]$ be a quasi-ordinary series 
    with characteristic exponents $\alpha_1, \dots, \alpha_g$. Let us denote by 
    $(e_0, \dots, e_d)$ the sequence of essential elements of the support $S(\psi)$ 
    relative to the lattice $\Z^h$ and a fixed additive order $\preceq$ dominating $\Q^h_+$.  
    Then:
\begin{itemize}
\item If $e_0 \notin \Z^h$, then $g = d+1$ and $\alpha_k = e_{k-1}$ for $k \in \{1, \dots, d+1\}$.
\item If $e_0 \in \Z^h$, then $g = d$ and $\alpha_k = e_k$ for $k \in \{ 1, \dots, d\}$.
\end{itemize} 
\end{lemma}

 \begin{remark}
Lipman proved an inversion theorem for the characteristic exponents of a quasi-ordinary 
series $\psi$, when $\psi$ is $x_1$-dominant. This result appeared in Lipman's PhD Thesis 
\cite[Lemma 2.3 and table 4.4]{L 65}, see also \cite{L 83}, while its proof 
was published later in \cite{L 88}. 
This proof is written in the two variable case but it extends naturally to more variables.
See also \cite[Proposition 5.5]{GP 03}.
Thanks to Corollary \ref{invPuiseuxgen} and Lemma \ref{charess-S}, we 
see that Lipman's inversion theorem for quasi-ordinary series is
a particular case of the part concerning exponents of  our inversion 
theorem (Corollary \ref{invPuiseuxgen}) for $x_1$-dominant Newton-Puiseux series 
in several variables.
\end{remark}
 
 \medskip

We end this paper with some remarks relating geometrically the Newton-Puiseux series with 
the quasi-ordinary series by using methods of toric geometry. They are inspired by  
the second-named author's proof of \cite[Th\'eor\`eme 3]{GP 00}.

  \medskip
 
 Let $f \in \C[[x_1, \dots, x_h]][y_1]$ be a reduced polynomial such that $f(0, \dots, 0) =0$. 
 Assume that $f$ is not quasi-ordinary. Then, the discriminant 
 $\Delta_{y_1} (f) \in \C [[x_1, \dots, x_h]]$ 
 is not of the form a monomial times a unit. It follows that the dual fan associated 
 to the Newton polyhedron of 
 $\Delta_{y_1} (f)$ defines a non-trivial subdivision of the positive quadrant $(\R^h)^{\vee}_+$ 
 of the vector space $(\R^h)^{\vee}$ of real weights of monomials 
 $x_1^{k_1} \cdots x_h^{k_h}$. 
 Let $\Sigma$ be a \emph{regular subdivision} of this dual fan. One has an associated  
 \emph{toric modification} $X_\Sigma \to \C^h$, which is obtained by patching  
 the monomial maps associated to the ${h}$-dimensional cones of 
 $\Sigma$. See for instance \cite{GT 00} or \cite{GP 00}
 for the basic definitions used in  these methods of toric geometry.

 Let $\sigma \in \Sigma$ be a ${h}$-dimensional cone of the fan $\Sigma$. It is spanned by 
 the forms $\gamma_1, \dots, \gamma_h $, 
 which are the primitive lattice vectors of the lattice $(\Z^h)^{\vee}$ lying
 on the edges of the cone $\sigma$. By the definition of the dual fan, the following property holds: 
 \begin{lemma} \label{Delta}
All the forms  $\gamma_1, \dots, \gamma_h$ reach
 their minimum value on the support of the discriminant 
$
 \Delta_{y_1} (f)
$
 at the same vertex $\lambda_0$ of its Newton polyhedron. 
 \end{lemma}

We consider the coordinates $  (\gamma_{s,1}, \dots, \gamma_{s,h})$ 
 of the vectors  $\gamma_s$,  $s \in \{ 1, \dots, h\}$, with respect to the dual basis of 
 $\nu_1, \dots, \nu_h$. Let $q_\sigma \in \mathrm{GL}(h, \Q)$ be the linear map  
defined, with respect to the canonical basis  $\nu_1, \dots, \nu_h$ of $\Q^h$,
by the matrix whose rows are $  (\gamma_{s,1}, \dots, \gamma_{s,h})$, for 
every $s \in \{ 1, \dots, h\}$.
 
\medskip

 The monomial map 
 \begin{equation} \label{chart}
 \C[x_1, \dots, x_h] \to \C[v_1, \dots, v_h],  \quad x^{\lambda} \to v^{q_\sigma (\lambda) }
 \end{equation}
defines the chart $\C^h_\sigma \to \C^h$
 of the toric modification of $X_\Sigma \to \C^h$, associated to cone $\sigma$. 
If: 
\[
F= \sum_{\gamma= (\alpha, \beta) \in \N^{h}\times \N} c_\gamma x^{\alpha}  y_1^{\beta} \in
 \C[[ x_1,\dots, x_h, y_1]]
 \] 
 then the  pull-back of $F$ on
$\C^h_\sigma \times \C$ is defined by:  
\[
F_\sigma =
 \sum_{\gamma= (\alpha, \beta) \in \N^{h}\times \N} c_\gamma v^{q_\sigma( \alpha ) }  y_1^{\beta} 
 \in \C[[ v_1, \dots, v_h, y_1]].
 \]

\medskip

Assume that  $\psi = \sum c_\lambda x^\lambda   \in \C[[ x_1^{1/ \N},\dots, x_h^{1/ \N} ]]$ is a Newton-Puiseux series of $f$. Then, by definition 
$\psi_\sigma:= \sum c_\lambda v^{q(\lambda)}  \in  \C[[ v_1^{1/ \N},\dots, v_h^{1/ \N} ]]$ is 
a Newton-Puiseux series of the pull-back $f_\sigma$ of $f$. 
In addition, $q(\lambda_0)$ belongs to the 
support of  the discriminant $\Delta_{y_1} (f_\sigma)$. By 
Lemma \ref{Delta},  if $q(\lambda)$ belongs to the support of $\Delta_{y_1} (f_\sigma)$, 
then $q(\lambda_0) \leq q(\lambda)$ (for the coordinate-wise order). This implies that 
 the Newton-Puiseux series $\psi_\sigma$ is quasi-ordinary, since the discriminant 
 $\Delta_{y_1} (f_\sigma)$  is of the form $v^{q_\sigma(Ê\lambda_0 )}$ times a unit.
Using Lemma \ref{ess-Pgen} and the fact that $\sigma$ is a regular cone 
(which implies that $q_\sigma( \Z^h) = \Z^h$), we obtain also the relation:
\begin{equation} \label{q_sigma}
q_\sigma \,  (\Ess (\psi, \Z^h, \preceq_{q_\sigma}) ) = \Ess ( \psi_\sigma , \Z^h, \preceq) 
\end{equation}
between the essential exponents of $\psi$ and 
the essential exponents of the quasi-ordinary series $\psi_\sigma$.

\begin{example} \label{E:Tor}
By Lemma \ref{l:Lipman}, the series $\psi = x_1^{3/2} + x_2^{1/4} + x_1^{7/2} \, x_2^{5/2}$ 
is not quasi-ordinary.
If we consider the chart of the blowing up of $0 \in \C^2$ given by 
$x_1 = v_1 v_2$ and $x_2 = v_2$, whose associated cone we denote by $\sigma$, 
then we obtain the series: 
\[
 \psi_\sigma = v_1^{3/2} \, v_2^{3/2} + v_2^{1/4} + v_1^{7/2} v_2^{6}.
\]
By Lemma \ref{l:Lipman},  the series   $\psi_\sigma$ 
is quasi-ordinary.  It has  essential exponents  $(0, 1/4)$  and $(3/2, 3/2)$ with respect to 
the lattice $\Z^2$ and any additive order $\preceq$ of $\Q^2$ dominating $\Q^2_+$.  
By Lemma \ref{charess-S}, 
these pairs are also the characteristic exponents of $\psi_\sigma$. 
By (\ref{q_sigma}), 
the pairs $(0, 1/4)$ and $(3/2,0)$ are the
essential exponents of $\psi$ with respect to the order $\preceq_{q_{\sigma}}$.
\end{example}

\begin{remark}  \label{R:Tor}
Tornero studied in \cite{T 08} a notion of \emph{distinguished exponents} of the 
Newton-Puiseux series 
$\psi  \in \C[[ x_1^{1/ \N},\dots, x_h^{1/ \N} ]]$ with respect to a fixed additive order 
$\preceq'$ of $\N^d$.  
In Example \ref{E:Tor},  one can check that 
the distinguished exponents of $\psi$ relative to the additive order $\preceq_{q_{\sigma}}$
correspond to the characteristic exponents  
of the quasi-ordinary series 
$\psi_\sigma$. One can prove that this is a general phenomenon. 
By the previous discussion, it is enough to show that given 
a fixed additive order  $\preceq'$,   there exists 
a unique $h$-dimensional cone $\sigma \in \Sigma$  
and a unique additive order $\preceq$ dominating $\Q^h_+$ such that the orders 
$\preceq'$ and $ \preceq_{q_\sigma}$ coincide.
Indeed, the cone $\sigma$ is the unique $h$-dimensional cone of $\Sigma$ 
such that the additive order 
$\preceq'$ dominates $\sigma^\vee \cap \Q^h$, where $\sigma^\vee$ is the dual cone of 
$\sigma$ (the existence of $\sigma$  is a consequence of the properties of the 
Zariski-Riemann space of the fan $\Sigma$, see 
\cite{EI 06} and  \cite[Section 3.5]{GT 14}). 
\end{remark}

\medskip
\end{document}